\documentclass[12pt,reqno]{amsart}

\newcommand{\df}{\dfrac}
\newcommand{\tf}{\tfrac}

\renewcommand{\(}{\left\(}
\renewcommand{\)}{\right\)}
\renewcommand{\[}{\left\[}
\renewcommand{\]}{\right\]}

\numberwithin{equation}{section}
 \theoremstyle{plain}
\newtheorem{theorem}{Theorem}[section]

\makeatletter
\def\proof{\@ifnextchar[{\@oproof}{\@nproof}}
\def\@oproof[#1][#2]{\trivlist\item[\hskip\labelsep\textit{#2 Proof of\
#1.}~]\ignorespaces}
\def\@nproof{\trivlist\item[\hskip\labelsep\textit{Proof.}~]\ignorespaces}

\makeatother

\begin{document}
\title{Analogues of the general theta transformation formula}
\author{Atul Dixit}\thanks{2010 \textit{Mathematics Subject Classification.} Primary 11M06, Secondary 33C15.\\
\textit{Keywords and phrases.} Riemann $\Xi$-function, hypergeometric function, theta transformation formula, Bessel function.}
\address{Department of Mathematics, University of Illinois, 1409 West Green Street, Urbana, IL 61801, USA} \email{aadixit2@illinois.edu}
\maketitle
\begin{abstract}
A new class of integrals involving the confluent hypergeometric function ${}_1F_{1}(a;c;z)$ and the Riemann $\Xi$-function is considered. It generalizes a class containing some integrals of S.~Ramanujan, G.H.~Hardy and W.L.~Ferrar and gives as by-products, transformation formulas of the form $F(z,\alpha)=F(iz,\beta)$, where $\alpha\beta=1$. As particular examples, we derive an extended version of the general theta transformation formula and generalizations of certain formulas of Ferrar and Hardy. A one-variable generalization of a well-known identity of Ramanujan is also given. We conclude with a generalization of a conjecture due to Ramanujan, Hardy and J.E.~Littlewood involving infinite series of M\"obius functions.
\end{abstract}
\section{Introduction}
For $\alpha\beta=\pi$, Re $\alpha^2, \beta^2>0$, the well-known transformation formula for the theta function $\varphi(q)=\sum_{n=-\infty}^{\infty}q^{n^2}$, where $|q|<1$, is given by \cite[p.~43, Entry 27(i)]{berndt0}
\begin{equation*}
\sqrt{\alpha}\varphi\left(e^{-\alpha^2}\right)=\sqrt{\beta}\varphi\left(e^{-\beta^2}\right).
\end{equation*} 
It can also be written alternatively, for $\alpha\beta=1$, in the form
\begin{align}\label{ttum}
\sqrt{\alpha}\bigg(\frac{1}{2\alpha}-\sum_{n=1}^{\infty}e^{-\pi\alpha^2n^2}\bigg)=\sqrt{\beta}\bigg(\frac{1}{2\beta}-\sum_{n=1}^{\infty}e^{-\pi\beta^2n^2}\bigg).
\end{align}
In \cite[p.~36]{titch}, one finds the following integral evaluation
\begin{equation}\label{titchint}
\int_{0}^{\infty}\frac{\Xi(t)}{t^2+\tfrac{1}{4}}\cos xt\, dt=\frac{\pi}{2}\left(e^{\frac{x}{2}}-2e^{-\frac{x}{2}}\sum_{n=1}^{\infty}e^{-\pi n^2e^{-2x}}\right),
\end{equation}
where $\Xi(t)$ is the Riemann $\Xi$-function defined by 
\begin{equation}\label{xif}
\Xi(t):=\xi(\tfrac{1}{2}+it),
\end{equation}
the $\xi(s)$ being the Riemann $\xi$-function
\begin{equation}\label{xi}
\xi(s):=\tfrac{1}{2}s(s-1)\pi^{-\tfrac{1}{2}s}\Gamma(\tfrac{s}{2})\zeta(s),
\end{equation}
and $\zeta(s)$ being the Riemann zeta function (see Section 2). Replacing $t$ by $t/2$ on the left-hand side of (\ref{titchint}), then setting $x=\log\alpha$ and simplifying, gives for $\alpha\beta=1$, 
\begin{equation}\label{titchint1}
\frac{2}{\pi}\int_{0}^{\infty}\frac{\Xi(t/2)}{1+t^2}\cos\bigg(\frac{1}{2}t\log\alpha\bigg)\, dt=\sqrt{\beta}\bigg(\frac{1}{2\beta}-\sum_{n=1}^{\infty}e^{-\pi\beta^2n^2}\bigg).
\end{equation}
Now the invariance of the integral on the left side of (\ref{titchint1}) under the map $\alpha\to\beta$ proves (\ref{ttum}). Such integrals involving the Riemann $\Xi$-function, which are invariant under certain maps, can be used to prove a variety of transformation formulas. A beautiful example illustrating this phenomenon is found on page 220 in Ramanujan's lost notebook \cite{lnb}, with the first proofs being given in \cite{bcbad}. This formula which Ramanujan describes as `curious' is given below.
\begin{theorem}\label{entry1} Define
\begin{equation*}\label{w1.27}
\lambda(x):=\psi(x)+\df{1}{2x}-\log x,
\end{equation*}
where
\begin{equation}\label{w1.15b}
\psi(x):=\df{\Gamma^\prime(x)}{\Gamma(x)}=-\gamma-\sum_{m=0}^{\infty}\left(\dfrac{1}{m+x}-\dfrac{1}{m+1}\right)
\end{equation}
is the logarithmic derivative of the Gamma function.
Let the Riemann's functions $\Xi(t)$ and $\xi(s)$ be defined as in \textup{(\ref{xif})} and \textup{(\ref{xi})} respectively.
If $\alpha$ and $\beta$ are positive numbers such that $\alpha\beta=1$, then
\begin{align*}
\sqrt{\alpha}\bigg(\frac{\gamma-\log(2\pi\alpha)}{2\alpha}&+\sum_{k=1}^{\infty}\lambda(k\alpha)\bigg)
=\sqrt{\beta}\bigg(\frac{\gamma-\log(2\pi\beta)}{2\beta}+\sum_{k=1}^{\infty}\lambda(k\beta)\bigg)\nonumber\\
&=-\frac{1}{\pi^{3/2}}\int_0^{\infty}\left|\Xi\left(\frac{1}{2}t\right)\Gamma\left(\frac{-1+it}{4}\right)\right|^2
\frac{\cos\left(\tfrac{1}{2}t\log\alpha\right)}{1+t^2}\, dt,
\end{align*}
where $\gamma$ denotes Euler's constant.
\end{theorem}
Ramanujan \cite{riemann} was the first one to employ the idea of using integrals involving the Riemann $\Xi$-function to prove transformation formulas. After Ramanujan, N.S.~Koshliakov made a fruitful use of this technique in several of his papers. Recently, this technique was further explored and extended by the author in \cite{series, dixit, transf, cartirxf} to obtain more general transformation formulas of the form $F(z,\alpha)=F(z,\beta)$ or their character analogues $F(z, \alpha,\chi)=F(-z, \beta,\overline{\chi})=F(-z,\alpha,\overline{\chi})=F(z,\beta,\chi)$, where $\alpha\beta=1$, in addition to the transformation formulas of the above type, i.e., of the form $F(\alpha)=F(\beta)$.

This paper focuses on formulas of the type $F(z,\alpha)=F(iz,\beta)$, where $\alpha\beta=1$. This work was motivated by the search for an integral representation, similar to (\ref{titchint}), for both sides of the following generalization of (\ref{ttum}), valid for $\alpha\beta=1$ and $z\in\mathbb{C}$,
\begin{align}\label{eqsym0}
\sqrt{\alpha}\left(\frac{e^{-\frac{z^2}{8}}}{2\alpha}-e^{\frac{z^2}{8}}\sum_{n=1}^{\infty}e^{-\pi\alpha^2n^2}\cos(\sqrt{\pi}\alpha nz)\right)
&=\sqrt{\beta}\left(\frac{e^{\frac{z^2}{8}}}{2\beta}-e^{-\frac{z^2}{8}}\sum_{n=1}^{\infty}e^{-\pi\beta^2n^2}\cos(i\sqrt{\pi}\beta nz)\right),
\end{align}
which is of the form $F(z,\alpha)=F(iz,\beta)$. In \cite[p.~252-253, Entry 7]{berndtn1}, this identity can be found in a slightly different form. Another version of (\ref{eqsym0}), given in terms of Ramanujan's theta function $f(a,b)=\sum_{n=-\infty}^{\infty}a^{n(n+1)/2}b^{n(n-1)/2}$, $|ab|<1$, and valid for $\alpha\beta=\pi$, is \cite[p.~36, Entry 20]{berndt0}
\begin{equation*}
e^{\tfrac{z^2}{4}}\sqrt{\alpha}f\left(e^{-\alpha^{2}+iz\alpha}, e^{-\alpha^{2}-iz\alpha}\right)=\sqrt{\beta}f\left(e^{-\beta^{2}+z\beta}, e^{-\beta^{2}-z\beta}\right).
\end{equation*}
For more details, see \cite{bgkt}.

Formulas of the type $F(z,\alpha)=F(iz,\beta)$ are generated through a one-variable generalization of integrals of the form $\int_{0}^{\infty}f\left(\frac{t}{2}\right)\Xi\left(\frac{t}{2}\right)\cos\left(\tf{1}{2}t\log\alpha\right)\, dt$, whose special cases were studied by Ramanujan, Koshliakov, G.H.~Hardy and W.L.~Ferrar. See \cite{series} for some examples. This one-variable generalization is of a different kind than those studied in \cite{transf, cartirxf} in that, the variable $z$ does not occur in the argument of the Riemann $\Xi$-function but rather in a function which generalizes the $\cos\left(\tfrac{1}{2}t\log\alpha\right)$ term. This function, which we denote by $\nabla(x,z,s)$, involves the confluent hypergeometric function ${}_1F_{1}(a;c;z)$ \cite[p.~188]{aar}, and is defined by
\begin{equation}\label{nabla}
\nabla(x,z,s):=\rho(x,z,s)+\rho(x,z,1-s),
\end{equation}
where 
\begin{align}\label{rho}
\rho(x,z,s)&:=x^{\frac{1}{2}-s}e^{-\frac{z^2}{8}}{}_1F_{1}\left(\frac{1-s}{2};\frac{1}{2};\frac{z^2}{4}\right),\nonumber\\
{}_1F_{1}(a;c;z)&=\sum_{n=0}^{\infty}\frac{(a)_{n}z^{n}}{(c)_{n}n!},
\end{align}
with $(a)_{n}$ being the rising factorial defined by
\begin{equation*}
(a)_{n}:=a(a+1)\cdots (a+n-1)=\frac{\Gamma(a+n)}{\Gamma(a)},
\end{equation*}
for $a\in\mathbb{C}$. It is easy to see that
\begin{equation*}
\nabla\left(\alpha,0,\frac{1+it}{2}\right)=\alpha^{-\frac{it}{2}}+\alpha^{\frac{it}{2}}=2\cos\left(\frac{1}{2}t\log\alpha\right).
\end{equation*}
The general form of the integrals giving rise to transformation formulas of the type $F(z,\alpha)=F(iz,\beta)$ is given by
\begin{equation}\label{sp0}
F(z,\alpha):=\int_{0}^{\infty}f\left(\frac{t}{2}\right)\Xi\left(\frac{t}{2}\right)\nabla\left(\alpha,z,\frac{1+it}{2}\right)\, dt,
\end{equation}
where $f(t)$ is of the form $f(t)=\phi(it)\phi(-it)$ and $\phi$ is analytic in $t$ as a function of a real variable.
To see this, recall Kummer's first transformation for the confluent hypergeometric function \cite[p.~191, Equation (4.1.11)]{aar}
\begin{equation}\label{kft}
{}_1F_{1}(a;c;z)=e^z{}_1F_{1}(c-a;c;-z).
\end{equation}
Using (\ref{kft}) in the second equality below and the fact that $\alpha\beta=1$, we see that
{\allowdisplaybreaks\begin{align}\label{sp1}
\nabla\left(\beta,iz,\frac{1+it}{2}\right)&=\rho\left(\beta,iz,\frac{1+it}{2}\right)+\rho\left(\beta,iz,\frac{1-it}{2}\right)\nonumber\\
&=\beta^{\tfrac{1}{2}-\tfrac{1+it}{2}}e^{\frac{z^2}{8}}{}_1F_{1}\left(\frac{1-it}{4};\frac{1}{2};-\frac{z^2}{4}\right)+\beta^{\tfrac{1}{2}-\tfrac{1-it}{2}}e^{\frac{z^2}{8}}{}_1F_{1}\left(\frac{1+it}{4};\frac{1}{2};-\frac{z^2}{4}\right)\nonumber\\
&=\alpha^{\tfrac{it}{2}}e^{-\frac{z^2}{8}}{}_1F_{1}\left(\frac{1+it}{4};\frac{1}{2};\frac{z^2}{4}\right)+\alpha^{-\tfrac{it}{2}}e^{-\frac{z^2}{8}}{}_1F_{1}\left(\frac{1-it}{4};\frac{1}{2};\frac{z^2}{4}\right)\nonumber\\
&=\rho\left(\alpha,z,\frac{1-it}{2}\right)+\rho\left(\alpha,z,\frac{1+it}{2}\right)\nonumber\\
&=\nabla\left(\alpha,z,\frac{1+it}{2}\right).
\end{align}}%
Then (\ref{sp0}) and (\ref{sp1}) imply that $F(z,\alpha)=F(iz,\beta)$.

We explicitly evaluate the integrals in (\ref{sp0}) for several choices of the function $f(t)$. These give rise to new analogues of the general theta transformation formula (\ref{eqsym0}). We begin by stating the general theta transformation formula itself obtained through such an integral. Its extended version is as follows.
\begin{theorem}\label{thetasym}
Let $z\in\mathbb{C}$. If $\alpha$ and $\beta$ are positive numbers such that $\alpha\beta=1$, then
\begin{align}\label{eqsym}
\sqrt{\alpha}\left(\frac{e^{-\frac{z^2}{8}}}{2\alpha}-e^{\frac{z^2}{8}}\sum_{n=1}^{\infty}e^{-\pi\alpha^2n^2}\cos(\sqrt{\pi}\alpha nz)\right)
&=\sqrt{\beta}\left(\frac{e^{\frac{z^2}{8}}}{2\beta}-e^{-\frac{z^2}{8}}\sum_{n=1}^{\infty}e^{-\pi\beta^2n^2}\cosh(\sqrt{\pi}\beta nz)\right)\nonumber\\
&=\frac{1}{\pi}\int_{0}^{\infty}\frac{\Xi(t/2)}{1+t^2}\nabla\left(\alpha,z,\frac{1+it}{2}\right)\, dt.
\end{align}
\end{theorem}
At the end of his short note \cite{ghh}, Hardy obtained an identity, whose corrected form (see for example, \cite{series}), is
\textit{\begin{equation}\label{hardyf}
\int_{0}^{\infty}\frac{\Xi(t/2)}{1+t^2}\frac{\cos nt}{\cosh \tf{1}{2}\pi t}\, dt = \frac{1}{4}e^{-n}\left(2n+\frac{1}{2}\gamma+\frac{1}{2}\log\pi+\log 2\right)+\tf{1}{2}e^{n}\int_{0}^{\infty}\psi(x+1)e^{-\pi x^2e^{4n}}\, dx,
\end{equation}
where $n$ is real and $\psi(x)$ is defined in \textup{(\ref{w1.15b})}.}

Later, Koshliakov \cite[Equations (14), (20)]{koshliakov5} expressed the above identity in a compact and symmetric form, which we rephrase in the following form, valid for $\alpha\beta=1$:
\begin{align}\label{hardy1f}
\sqrt{\alpha}\int_{0}^{\infty}\left(\psi(x+1)-\log x\right)e^{-\pi \alpha^2x^2}\, dx&=\sqrt{\beta}\int_{0}^{\infty}\left(\psi(x+1)-\log x\right)e^{-\pi \beta^2x^2}\, dx\nonumber\\
&=2\int_{0}^{\infty}\frac{\Xi(t/2)}{1+t^2}\frac{\cos \left(\tf{1}{2}t\log\alpha\right)}{\cosh \tf{1}{2}\pi t}\, dt.
\end{align}
This is seen at once by letting $n=\tfrac{1}{2}\log\alpha$ in (\ref{hardyf}) and by using the formula \cite[p.~572, formula 4.333]{grn} $\int_{0}^{\infty}e^{-\pi\alpha^2x^2}\log x\, dx=-\tfrac{1}{4\alpha}\left(\gamma+\log\left(4\alpha^2\pi\right)\right)$. Koshliakov also proved this identity in several of his other papers, namely,  \cite[Equation 30.5]{koshliakov6p2}, \cite[Equation 34.10]{koshliakov6p3}\footnote{The formula here contains a typo, as the factor $\frac{1}{2}\log 2\pi$ should be $\frac{1}{2}\log\pi$.}, \cite[Equations 18, 19]{koshliakov7}. He also gave two different generalizations of Hardy's formula; one in \cite[Equation 30.4]{koshliakov6p2} and \cite[Equation 34.1]{koshliakov6p3}, and another in \cite[Equation 27]{koshliakov3}. (See \cite[p.~198--199]{bfikm} for the genesis of the monograph \cite{koshliakov3} written under Koshliakov's patronymic name `N. S. Sergeev'.)

Here, we obtain the following new generalization of (\ref{hardy1f}), again of the form $F(z,\alpha)=F(iz,\beta)$.
\begin{theorem}
Let $z\in\mathbb{C}$ and let $\psi(x)$ be defined as in \textup{(\ref{w1.15b})}. If $\alpha$ and $\beta$ are two positive numbers such that $\alpha\beta=1$, then
\begin{align}\label{hfg}
&\sqrt{\alpha}e^{\frac{z^2}{8}}\int_{0}^{\infty}\left(\psi(x+1)-\log x\right)e^{-\pi \alpha^2x^2}\cos\left(\sqrt{\pi}\alpha xz\right)\, dx\nonumber\\
&=\sqrt{\beta}e^{-\frac{z^2}{8}}\int_{0}^{\infty}\left(\psi(x+1)-\log x\right)e^{-\pi \beta^2x^2}\cosh\left(\sqrt{\pi}\beta xz\right)\, dx\nonumber\\
&=\int_{0}^{\infty}\frac{\Xi(t/2)}{1+t^2}\frac{\nabla\left(\alpha, z, \frac{1+it}{2}\right)}{\cosh \tf{1}{2}\pi t}\, dt.
\end{align}
\end{theorem}
In \cite{ferrar}, Ferrar obtained some transformation formulas of the form $F(\alpha)=F(\beta)$, of which one is rephrased in the form given below.\\

\textit{
Let $K_{0}(z)$ denote the modified Bessel function of order $0$. If $\alpha$ and $\beta$ are positive numbers such that $\alpha\beta=1$, then
{\allowdisplaybreaks\begin{align}\label{ferex}
&\sqrt{\alpha}\left(\frac{-\gamma+\log 16\pi+2\log\alpha}{\alpha}-2\sum_{n=1}^{\infty}\left(e^{\frac{\pi \alpha^2n^2}{2}}K_{0}\left(\frac{\pi \alpha^2n^2}{2}\right)-\frac{1}{n\alpha}\right)\right)\nonumber\\
&=\sqrt{\beta}\left(\frac{-\gamma+\log 16\pi+2\log\beta}{\beta}-2\sum_{n=1}^{\infty}\left(e^{\frac{\pi\beta^2n^2}{2}}K_{0}\left(\frac{\pi \beta^2n^2}{2}\right)-\frac{1}{n\beta}\right)\right)\nonumber\\
&=4\pi^{-\frac{3}{2}}\int_{0}^{\infty}\Gamma\left(\frac{1+it}{4}\right)\Gamma\left(\frac{1-it}{4}\right)\Xi\left(\frac{t}{2}\right)\frac{\cos\left(\frac{1}{2}t\log\alpha\right)}{1+t^2}\, dt.
\end{align}}}%
This also admits the following new generalization, which is the third example of the form $F(z,\alpha)=F(iz,\beta)$.
\begin{theorem}
Let $z\in\mathbb{C}$ and let $K_{0}(z)$ be defined above. If $\alpha$ and $\beta$ are positive numbers such that $\alpha\beta=1$, then
\begin{align}\label{gf}
&\sqrt{\alpha}e^{\frac{z^2}{8}}\int_{0}^{\infty}e^{-\frac{\alpha^2t^2}{4\pi}}\cos\left(\frac{\alpha tz}{2\sqrt{\pi}}\right)\left(\sum_{n=1}^{\infty}K_{0}(nt)-\frac{\pi}{2t}\right)\, dt\nonumber\\
&=\sqrt{\beta}e^{\frac{-z^2}{8}}\int_{0}^{\infty}e^{-\frac{\beta^2t^2}{4\pi}}\cosh\left(\frac{\beta tz}{2\sqrt{\pi}}\right)\left(\sum_{n=1}^{\infty}K_{0}(nt)-\frac{\pi}{2t}\right)\, dt\nonumber\\
&=-\frac{1}{2\sqrt{\pi}}\int_{0}^{\infty}\Gamma\left(\frac{1+it}{4}\right)\Gamma\left(\frac{1-it}{4}\right)\frac{\Xi(t/2)}{1+t^2}\nabla\left(\alpha,z,\frac{1+it}{2}\right)\, dt.\nonumber\\
\end{align}
\end{theorem}
For real $n$, Ramanujan \cite{riemann} showed that
\begin{equation}\label{rampe}
e^{-n}-4\pi e^{-3n}\int_{0}^{\infty}\frac{xe^{-\pi x^2e^{-4n}}}{e^{2\pi x}-1}\, dx=\frac{1}{4\pi\sqrt{\pi}}\int_{0}^{\infty}\Gamma\left(\frac{-1+it}{4}\right)\Gamma\left(\frac{-1-it}{4}\right)\Xi\left(\frac{t}{2}\right)\cos nt\, dt.
\end{equation}
As noted in \cite{series}, letting $n=\tfrac{1}{2}\log\alpha$ in the above identity and noting that the resulting integral on the left side is invariant under $\alpha\to\beta$, where $\alpha\beta=1$, gives
\begin{align}\label{mrram}
&\alpha^{-\frac{1}{2}}-4\pi\alpha^{-\frac{3}{2}}\int_{0}^{\infty}\frac{xe^{-\frac{\pi x^2}{\alpha^2}}}{e^{2\pi x}-1}\, dx=\beta^{-\frac{1}{2}}-4\pi\beta^{-\frac{3}{2}}\int_{0}^{\infty}\frac{xe^{-\frac{\pi x^2}{\beta^2}}}{e^{2\pi x}-1}\, dx\nonumber\\
&=\frac{1}{4\pi\sqrt{\pi}}\int_{0}^{\infty}\Gamma\left(\frac{-1+it}{4}\right)\Gamma\left(\frac{-1-it}{4}\right)\Xi\left(\frac{t}{2}\right)\cos \left(\frac{1}{2}t\log\alpha\right)\, dt.
\end{align}
For more details on the first equality in (\ref{mrram}), see \cite{series}. Interestingly, this identity does not admit a generalization of the form $F(z,\alpha)=F(iz,\beta)$, as can be seen from the following new generalization of (\ref{rampe}).
\begin{theorem}\label{thmrgenr}
Let $z\in\mathbb{C}$ and let $\rho(x,z,s)$ be defined in \textup{(\ref{rho})}. Then,
\begin{align}\label{rgenr}
&\alpha^{-\frac{1}{2}}e^{-\frac{z^2}{8}}-4\pi\alpha^{\frac{1}{2}}e^{\frac{z^2}{8}}\int_{0}^{\infty}\frac{xe^{-\pi\alpha^2x^2}\cos\left(\sqrt{\pi}\alpha xz\right)}{e^{2\pi x}-1}\, dx\nonumber\\
&=\frac{1}{8\pi^{3/2}}\int_{-\infty}^{\infty}\Gamma\left(\frac{-1+it}{4}\right)\Gamma\left(\frac{-1-it}{4}\right)\Xi\left(\frac{t}{2}\right)\rho\left(\alpha,z,\frac{3+it}{2}\right)\, dt.
\end{align}
\end{theorem}
Obviously, the presence of $\tfrac{3}{2}$ instead of $\tfrac{1}{2}$ in $\rho\left(\alpha,z,\frac{3+it}{2}\right)$ destroys the invariance property under the simultaneous application of the maps $\alpha\to\beta$ and $z\to iz$. Regarding the special case when $z=0$ of the integral on the right-hand side of (\ref{rgenr}) (i.e., the integral on the right-hand side of (\ref{mrram})), Hardy says in \cite{ghh}, ``The integral has properties similar to those of the integral by means of which I proved recently that $\zeta(s)$ has an infinity of zeros on the line $\sigma=\frac{1}{2}$ and may be used for the same purpose.'' It may be interesting to see what information can be extracted from the general integral.\\

In \cite[p.~156, Section 2.5]{hl}, Hardy and Littlewood discuss the following amazing identity, actually a conjecture, involving infinite series of M\"{o}bius functions having its genesis in the work of Ramanujan \cite[p.~470]{berndt1}.\\

\textit{Let $\mu(n)$ denote the M\"obius function. Let $\alpha$ and $\beta$ be two positive numbers such that $\alpha\beta=1$. Assume that the series $\sum_{\rho}\left(\Gamma{\left(\frac{1-\rho}{2}\right)}/\zeta^{'}(\rho)\right)a^{\rho}$ converges, where $\rho$ runs through the non-trivial zeros of $\zeta(s)$ and $a$ denotes a positive real number, and that the non-trivial zeros of $\zeta(s)$ are simple. Then
\begin{align}\label{mr}
&\sqrt{\alpha}\sum_{n=1}^{\infty}\frac{\mu(n)}{n}e^{-\frac{\pi\alpha^2}{n^2}}-\frac{1}{4\sqrt{\pi}\sqrt{\alpha}}\sum_{\rho}\frac{\Gamma{\left(\frac{1-\rho}{2}\right)}}{\zeta^{'}(\rho)}\pi^{\frac{\rho}{2}}\alpha^{\rho}\nonumber\\
&=\sqrt{\beta}\sum_{n=1}^{\infty}\frac{\mu(n)}{n}e^{-\frac{\pi\beta^2}{n^2}}-\frac{1}{4\sqrt{\pi}\sqrt{\beta}}\sum_{\rho}\frac{\Gamma{\left(\frac{1-\rho}{2}\right)}}{\zeta^{'}(\rho)}\pi^{\frac{\rho}{2}}\beta^{\rho}.
\end{align}
}%
The original formulation, slightly different in \cite{hl}, can be easily seen to be equivalent to (\ref{mr}). See also \cite[p.~143]{kp} and \cite[p.~219, Section 9.8]{titch} for discussions of this identity. The above conjecture admits the following generalization, also of the form $F(z,\alpha)=F(iz,\beta)$.
\begin{theorem}\label{rhlg}
Let $\mu(n)$ denote the M\"obius function. Let $z\in\mathbb{C}$ and let $\alpha$ and $\beta$ be two positive numbers such that $\alpha\beta=1$. Assume that the series $\sum_{\rho}\frac{\Gamma{\left(\frac{1-\rho}{2}\right)}}{\zeta^{'}(\rho)}{}_1F_1\left(\frac{1-\rho}{2};\frac{1}{2};\frac{-z^2}{4}\right)\pi^{\frac{\rho}{2}}a^{\rho}$ converges, where $\rho$ runs through the non-trivial zeros of $\zeta(s)$ and $a$ denotes a positive real number, and that the non-trivial zeros of $\zeta(s)$ are simple. Then
\begin{align}\label{mrg}
&\sqrt{\alpha}e^{\frac{z^2}{8}}\sum_{n=1}^{\infty}\frac{\mu(n)}{n}e^{-\frac{\pi\alpha^2}{n^2}}\cos\left(\frac{\sqrt{\pi}\alpha z}{n}\right)-\frac{e^{\frac{z^2}{8}}}{4\sqrt{\pi}\sqrt{\alpha}}\sum_{\rho}\frac{\Gamma{\left(\frac{1-\rho}{2}\right)}}{\zeta^{'}(\rho)}{}_1F_1\left(\frac{1-\rho}{2};\frac{1}{2};-\frac{z^2}{4}\right)\pi^{\frac{\rho}{2}}\alpha^{\rho}\nonumber\\
&=\sqrt{\beta}e^{-\frac{z^2}{8}}\sum_{n=1}^{\infty}\frac{\mu(n)}{n}e^{-\frac{\pi\beta^2}{n^2}}\cosh\left(\frac{\sqrt{\pi}\beta z}{n}\right)-\frac{e^{-\frac{z^2}{8}}}{4\sqrt{\pi}\sqrt{\beta}}\sum_{\rho}\frac{\Gamma{\left(\frac{1-\rho}{2}\right)}}{\zeta^{'}(\rho)}{}_1F_1\left(\frac{1-\rho}{2};\frac{1}{2};\frac{z^2}{4}\right)\pi^{\frac{\rho}{2}}\beta^{\rho}.
\end{align}
\end{theorem}
This paper is organized as follows. In Section 2, we discuss preliminary results which are subsequently used in the later sections. In Section 3, we obtain a line integral representation for the integral in (\ref{sp0}). Then in Sections 4, 5, 6 and 7, we prove Theorems 1.2, 1.3, 1.4 and 1.5 respectively. In Section 8, we give proof of Theorem 1.6. Finally, we conclude the paper with some remarks on further developments that may be possible.
\section{Preliminary results}
The Riemann zeta function $\zeta(s)$ is defined for Re $s>1$ by the absolutely convergent Dirichlet series
\begin{equation}\label{zzdefgr}
\zeta(s)=\sum_{m=1}^{\infty}\frac{1}{m^{s}}.
\end{equation}
It can be analytically continued first to $0<$ Re $s<1$ by an elementary argument and then to the whole complex plane, except for a simple pole at $s=1$, by means of the following functional equation \cite[p.~22, eqn. (2.6.4)]{titch}
\begin{equation}\label{zetafe}
\pi^{-\frac{s}{2}}\Gamma\left(\frac{s}{2}\right)\zeta(s)=\pi^{-\frac{(1-s)}{2}}\Gamma\left(\frac{1-s}{2}\right)\zeta(1-s),
\end{equation}
which can also be written in the form 
\begin{equation}\label{zetaalt}
\xi(s)=\xi(1-s),
\end{equation}
where $\xi(s)$ is the Riemann $\xi$-function defined in (\ref{xi}). We also need some basic properties of the Gamma function $\Gamma(s)$. 
The reflection formula for the Gamma function \cite[p.~46]{temme} is given by
\begin{equation}\label{rf}
\displaystyle\Gamma(s)\Gamma(1-s)=\frac{\pi}{\sin \pi s},
\end{equation}
for $s\notin\mathbb{Z}$. Further, Legendre's duplication formula \cite[p.~46]{temme} gives
\begin{equation}\label{dup}
\Gamma(s)\Gamma\left(s+\frac{1}{2}\right)=\frac{\sqrt{\pi}}{2^{2s-1}}\Gamma(2s),
\end{equation}
Stirling's formula for $\Gamma(s)$, $s=\sigma+it$, in a vertical strip $\alpha\leq\sigma\leq\beta$ is given by
\begin{equation}\label{strivert}
|\Gamma(s)|=(2\pi)^{\tf{1}{2}}|t|^{\sigma-\tf{1}{2}}e^{-\tf{1}{2}\pi |t|}\left(1+O\left(\frac{1}{|t|}\right)\right),
\end{equation}
as $|t|\to\infty$. We will also require the inverse Mellin transform representation of the function $e^{-ax^2}\cos bx$ \cite[p.~47, Equation 5.30]{ober}, valid for $c=$ Re $s>0$, and given by
\begin{equation}\label{invmel}
e^{-ax^2}\cos bx=\frac{1}{2\pi i}\int_{c-i\infty}^{c+i\infty}\frac{1}{2}a^{-\tfrac{s}{2}}\Gamma\left(\frac{s}{2}\right)e^{-\frac{b^2}{4a}}{}_1F_{1}\left(\frac{1-s}{2};\frac{1}{2};\frac{b^2}{4a}\right)x^{-s}\, ds,
\end{equation}
which can be easily proved by employing the series representation of ${}_1F_{1}$ and then interchanging the order of summation and integration. Finally, we require the asymptotic expansion, for large values of $|\lambda|$, of the Whittaker function $M_{\lambda,\mu}(z)$ defined by \cite[p.~1024, formula 9.220, no. 2]{grn}
\begin{equation}\label{whi}
M_{\lambda,\mu}(z)=z^{\mu+\frac{1}{2}}e^{-z/2}{}_1F_{1}\left(\mu-\lambda+\tfrac{1}{2};2\mu+1;z\right).
\end{equation}
Its asymptotic expansion \cite[p.~1026, formula 9.228]{grn} is given by
\begin{equation}\label{whias}
M_{\lambda,\mu}(z)\sim \frac{1}{\sqrt{\pi}}\Gamma(2\mu+1)\lambda^{-\mu-\frac{1}{4}}z^{1/4}\cos\left(2\sqrt{\lambda z}-\mu\pi-\frac{\pi}{4}\right),
\end{equation}
as $|\lambda|\to\infty$. Letting $\mu=-\frac{1}{4}$ and replacing $z$ by $z^2/4$ in (\ref{whias}) and using (\ref{whi}), we obtain, upon simplification,
\begin{equation}\label{conasy}
{}_1F_{1}\left(\tfrac{1}{4}-\lambda;\tfrac{1}{2};\tfrac{z^{2}}{4}\right)\sim e^{z^2/8}\cos\left(\sqrt{\lambda}z\right),
\end{equation}
as $|\lambda|\to\infty$.
\section{A line integral representation}
Here we give a line integral representation for the integral in (\ref{sp0}) that will allow us to use the residue theorem and Mellin transforms for its evaluation. 
\begin{theorem}
Let 
\begin{equation*}
f(t)=\phi(it)\phi(-it),
\end{equation*} 
where $\phi$ is analytic in $t$ as a function of a real variable. Let $\nabla(x,z,s)$ and $\rho(x,z,s)$ be defined as in \textup{(\ref{nabla})} and \textup{(\ref{rho})}. Assume that the integral on the left side below converges. Then,
\begin{equation}\label{sp3}
\int_{0}^{\infty}f(t)\Xi(t)\nabla\left(\alpha,z,\frac{1}{2}+it\right)\, dt
=\frac{1}{i}\int_{\frac{1}{2}-i\infty}^{\frac{1}{2}+i\infty}\phi\left(s-\frac{1}{2}\right)\phi\left(\frac{1}{2}-s\right)\xi(s)\rho(\alpha,z,s)\, ds.
\end{equation}
\end{theorem}
\begin{proof}
Let $I(z,\alpha)$ denote the left-hand side of (\ref{sp3}). Then using the facts that $f(t), \Xi(t)$ and $\nabla\left(\alpha,z,\frac{1}{2}+it\right)$ are all even functions of $t$, we have
{\allowdisplaybreaks\begin{align}\label{enta1}
I(z,\alpha)&=\frac{1}{2}\left(\int_{0}^{\infty}f(t)\Xi(t)\nabla\left(\alpha,z,\frac{1}{2}+it\right)\, dt-\int_{0}^{-\infty}f(-t)\Xi(-t)\nabla\left(\alpha,z,\frac{1}{2}-it\right)\, dt\right)\nonumber\\
&=\frac{1}{2}\int_{-\infty}^{\infty}f(t)\Xi(t)\nabla\left(\alpha,z,\frac{1}{2}+it\right)\, dt\nonumber\\
&=\frac{1}{2i}\int_{\frac{1}{2}-i\infty}^{\frac{1}{2}+i\infty}\phi\left(s-\frac{1}{2}\right)\phi\left(\frac{1}{2}-s\right)\xi(s)\left(\rho(\alpha,z,s)+\rho(\alpha,z,1-s)\right)\, ds\nonumber\\
&=\frac{1}{2i}\left(I_{1}(z,\alpha)+I_{2}(z,\alpha)\right),
\end{align}}
where 
\begin{align*}
I_{1}(z,\alpha)&:=\int_{\frac{1}{2}-i\infty}^{\frac{1}{2}+i\infty}\phi\left(s-\frac{1}{2}\right)\phi\left(\frac{1}{2}-s\right)\xi(s)\rho(\alpha,z,s)\, ds\nonumber\\
I_{2}(z,\alpha)&:=\int_{\frac{1}{2}-i\infty}^{\frac{1}{2}+i\infty}\phi\left(s-\frac{1}{2}\right)\phi\left(\frac{1}{2}-s\right)\xi(s)\rho(\alpha,z,1-s)\, ds,\nonumber\\
\end{align*}
and in the penultimate step in (\ref{enta1}), we have performed a change of variable $s=\frac{1}{2}+it$. Now rewriting $I_{2}(z,\alpha)$ by employing (\ref{zetaalt}), and then replacing $s$ by $1-s$, we easily see that
\begin{align}\label{enta2}
I_{2}(z,\alpha)&=\int_{\frac{1}{2}-i\infty}^{\frac{1}{2}+i\infty}\phi\left(\frac{1}{2}-(1-s)\right)\phi\left(1-s-\frac{1}{2}\right)\xi(1-s)\rho(\alpha,z,1-s)\, ds,\nonumber\\
&=\int_{\frac{1}{2}-i\infty}^{\frac{1}{2}+i\infty}\phi\left(s-\frac{1}{2}\right)\phi\left(\frac{1}{2}-s\right)\xi(s)\rho(\alpha,z,s)\, ds\nonumber\\
&=I_{1}(z,\alpha).
\end{align}
Hence, substituting (\ref{enta2}) in (\ref{enta1}), we obtain (\ref{sp3}).
\end{proof}
For our purposes, we will use the following alternative form of (\ref{sp3}), which is easily obtained by replacing $t$ by $t/2$ on the left-hand side of (\ref{sp3}):
\begin{equation}\label{sp4}
\int_{0}^{\infty}f\left(\frac{t}{2}\right)\Xi\left(\frac{t}{2}\right)\nabla\left(\alpha,z,\frac{1+it}{2}\right)\, dt
=\frac{2}{i}\int_{\frac{1}{2}-i\infty}^{\frac{1}{2}+i\infty}\phi\left(s-\frac{1}{2}\right)\phi\left(\frac{1}{2}-s\right)\xi(s)\rho(\alpha,z,s)\, ds.
\end{equation}
\section{Proof of Theorem \ref{thetasym}: Extended version of the general theta transformation formula}
Using (\ref{xif}), (\ref{xi}), (\ref{nabla}), (\ref{rho}), (\ref{strivert}) and (\ref{conasy}), we easily see that the integral on the extreme right-hand side of (\ref{eqsym}) converges. Let $\phi(t)=\frac{1}{t+1/2}$ so that $f(t)=\phi(it)\phi(-it)=\frac{1}{t^{2}+1/4}$. Substituting this in (\ref{sp4}) and using (\ref{xi}) and (\ref{rho}), we have
{\allowdisplaybreaks\begin{align}\label{sp4.5}
&4\int_{0}^{\infty}\frac{\Xi(t/2)}{1+t^2}\nabla\left(\alpha,z,\frac{1+it}{2}\right)\, dt\nonumber\\
&=i\int_{\frac{1}{2}-i\infty}^{\frac{1}{2}+i\infty}\pi^{-\frac{s}{2}}\Gamma\left(\frac{s}{2}\right)\zeta(s)\rho(\alpha,z,s)\, ds\nonumber\\
&=i\alpha^{\frac{1}{2}}e^{-\frac{z^2}{8}}\int_{\frac{1}{2}-i\infty}^{\frac{1}{2}+i\infty}\Gamma\left(\frac{s}{2}\right)\zeta(s){}_1F_{1}\left(\frac{1-s}{2};\frac{1}{2};\frac{z^2}{4}\right)(\sqrt{\pi}\alpha)^{-s}\, ds.
\end{align}}

To evaluate the last integral, we shift the line of integration from Re $s=1/2$ to Re $s=1+\delta$, $\delta>0$, so that we can use (\ref{zzdefgr}). Consider a positively oriented rectangular contour with sides $[\frac{1}{2}+iT, \frac{1}{2}-iT], [\frac{1}{2}-iT, 1+\delta-iT], [1+\delta-iT,1+\delta+iT]$ and $[1+\delta+iT,\frac{1}{2}+iT]$, where $T$ is any positive real number. While shifting, we encounter the pole of the integrand at $s=1$. Hence, using the residue theorem, we have
{\allowdisplaybreaks\begin{align}\label{resimpact}
&\int_{\frac{1}{2}-iT}^{\frac{1}{2}+iT}\Gamma\left(\frac{s}{2}\right)\zeta(s){}_1F_{1}\left(\frac{1-s}{2};\frac{1}{2};\frac{z^2}{4}\right)(\sqrt{\pi}\alpha)^{-s}\, ds\nonumber\\
&=\left[\int_{\tf{1}{2}-iT}^{1+\delta-iT}+\int_{1+\delta-iT}^{1+\delta+iT}+\int_{1+\delta+iT}^{\tf{1}{2}+iT}\right]\Gamma\left(\frac{s}{2}\right)\zeta(s){}_1F_{1}\left(\frac{1-s}{2};\frac{1}{2};\frac{z^2}{4}\right)(\sqrt{\pi}\alpha)^{-s}\, ds\nonumber\\
&\quad-2\pi i\lim_{s\to 1}(s-1)\Gamma\left(\frac{s}{2}\right)\zeta(s){}_1F_{1}\left(\frac{1-s}{2};\frac{1}{2};\frac{z^2}{4}\right)(\sqrt{\pi}\alpha)^{-s}.
\end{align}}%
Using (\ref{strivert}), one easily sees that the integrals on the horizontal segments $[\tf{1}{2}-iT,1+\delta-iT]$ and $[1+\delta+iT, \tf{1}{2}+iT]$ tend to zero as $T\to\infty$. Also,
\begin{align}\label{calres}
\lim_{s\to 1}(s-1)\Gamma\left(\frac{s}{2}\right)\zeta(s){}_1F_{1}\left(\frac{1-s}{2};\frac{1}{2};\frac{z^2}{4}\right)(\sqrt{\pi}\alpha)^{-s}=\frac{1}{\alpha}.
\end{align}
Therefore, letting $T\to\infty$ in (\ref{resimpact}), using (\ref{calres}) and (\ref{zzdefgr}) in the integral over $[1+\delta-i\infty,1+\delta+i\infty]$, we have
\begin{align}\label{resconc}
&\int_{\frac{1}{2}-i\infty}^{\frac{1}{2}+i\infty}\Gamma\left(\frac{s}{2}\right)\zeta(s){}_1F_{1}\left(\frac{1-s}{2};\frac{1}{2};\frac{z^2}{4}\right)(\sqrt{\pi}\alpha)^{-s}\, ds\nonumber\\
&=\int_{1+\delta-i\infty}^{1+\delta+i\infty}\sum_{n=1}^{\infty}\Gamma\left(\frac{s}{2}\right){}_1F_{1}\left(\frac{1-s}{2};\frac{1}{2};\frac{z^2}{4}\right)(\sqrt{\pi}\alpha n)^{-s}\, ds-\frac{2\pi i}{\alpha}\nonumber\\
&=\sum_{n=1}^{\infty}\int_{1+\delta-i\infty}^{1+\delta+i\infty}\Gamma\left(\frac{s}{2}\right){}_1F_{1}\left(\frac{1-s}{2};\frac{1}{2};\frac{z^2}{4}\right)(\sqrt{\pi}\alpha n)^{-s}\, ds-\frac{2\pi i}{\alpha},
\end{align}
where in the last step, we have interchanged the order of summation and integration, which is valid because of absolute convergence.

Letting $a=1,x=\sqrt{\pi}\alpha n, b=z$ in (\ref{invmel}), we see that
\begin{equation}\label{invmel1}
\int_{1+\delta-i\infty}^{1+\delta+i\infty}\Gamma\left(\frac{s}{2}\right){}_1F_{1}\left(\frac{1-s}{2};\frac{1}{2};\frac{z^2}{4}\right)(\sqrt{\pi}\alpha n)^{-s}\, ds=4\pi ie^{-\pi\alpha^2n^2+z^2/4}\cos\left(\sqrt{\pi}\alpha nz\right).
\end{equation}
Now (\ref{sp4.5}), (\ref{resconc}) and (\ref{invmel1}) imply that
\begin{equation}\label{fingtf}
\frac{1}{\pi}\int_{0}^{\infty}\frac{\Xi(t/2)}{1+t^2}\nabla\left(\alpha,z,\frac{1+it}{2}\right)\, dt=\sqrt{\alpha}\left(\frac{e^{-\frac{z^2}{8}}}{2\alpha}-e^{\frac{z^2}{8}}\sum_{n=1}^{\infty}e^{-\pi\alpha^2n^2}\cos(\sqrt{\pi}\alpha nz)\right).
\end{equation}
Finally, replacing $z$ by $iz$ and $\alpha$ by $\beta$ in (\ref{fingtf}), noting that the integral on the left-hand side remains invariant in this process, and then combining the result with (\ref{fingtf}), we arrive at (\ref{eqsym}).
\section{Generalization of Hardy's formula}
Let $\phi(s)=\frac{1}{4\sqrt{2}\pi}\Gamma\left(\frac{1}{4}+\frac{s}{2}\right)\Gamma\left(\frac{-1}{4}+\frac{s}{2}\right)$, so that 
\begin{equation*}
f(t)=\phi(it)\phi(-it)=\frac{1}{32\pi^2}\Gamma\left(\frac{1}{4}+\frac{it}{2}\right)\Gamma\left(\frac{-1}{4}+\frac{it}{2}\right)\Gamma\left(\frac{1}{4}-\frac{it}{2}\right)\Gamma\left(\frac{-1}{4}-\frac{it}{2}\right).
\end{equation*}
Applying (\ref{rf}) twice, we easily see that $f\left(\frac{t}{2}\right)=1/((1+t^2)\cosh \tfrac{1}{2}\pi t)$. Using these facts along with (\ref{xif}), (\ref{xi}), (\ref{nabla}), (\ref{rho}), (\ref{strivert}) and (\ref{conasy}), we find that the integral on the extreme right-hand side of (\ref{hfg}) converges. Substituting this in (\ref{sp4}) and using (\ref{xi}) and (\ref{rho}), we have
\begin{align}\label{nsphar}
&\int_{0}^{\infty}\frac{\Xi(\tf{1}{2}t)}{1+t^2}\frac{\nabla\left(\alpha, z, \frac{1+it}{2}\right)}{\cosh \tf{1}{2}\pi t}\, dt\nonumber\\
&=\frac{1}{16\pi^2 i}\int_{\frac{1}{2}-i\infty}^{\frac{1}{2}+i\infty}\pi^{-\frac{s}{2}}\frac{s(s-1)}{2}\Gamma^{2}\left(\frac{s}{2}\right)\Gamma\left(\frac{s-1}{2}\right)\Gamma\left(\frac{1-s}{2}\right)\Gamma\left(\frac{-s}{2}\right)\zeta(s)\rho(\alpha,z,s)\, ds\nonumber\\
&=-\frac{\alpha^{\frac{1}{2}}e^{-\frac{z^2}{8}}}{4i}\int_{\frac{1}{2}-i\infty}^{\frac{1}{2}+i\infty}\Gamma\left(\frac{s}{2}\right)\frac{\zeta(s)}{\sin\pi s}{}_1F_{1}\left(\frac{1-s}{2};\frac{1}{2};\frac{z^2}{4}\right)(\sqrt{\pi}\alpha)^{-s}\, ds,
\end{align}
where in the last step, we used (\ref{dup}) as well as (\ref{rf}).

To evaluate the last integral, we shift the line of integration from Re $s=1/2$ to Re $s=1+\delta$, $\delta>0$, so that we can use (\ref{zzdefgr}). Consider a positively oriented rectangular contour with sides $[\frac{1}{2}+iT, \frac{1}{2}-iT], [\frac{1}{2}-iT, 1+\delta-iT], [1+\delta-iT,1+\delta+iT]$ and $[1+\delta+iT,\frac{1}{2}+iT]$, where $T$ is any positive real number. While shifting the line of integration, we encounter the pole of order two of the integrand (due to $\zeta(s)$ and $\sin\pi s$) at $s=1$. Hence, using the residue theorem, we have
\begin{align}\label{resimpact2}
&\int_{\frac{1}{2}-iT}^{\frac{1}{2}+iT}\Gamma\left(\frac{s}{2}\right)\frac{\zeta(s)}{\sin\pi s}{}_1F_{1}\left(\frac{1-s}{2};\frac{1}{2};\frac{z^2}{4}\right)(\sqrt{\pi}\alpha)^{-s}\, ds\nonumber\\
&=\left[\int_{\tf{1}{2}-iT}^{1+\delta-iT}+\int_{1+\delta-iT}^{1+\delta+iT}+\int_{1+\delta+iT}^{\tf{1}{2}+iT}\right]\Gamma\left(\frac{s}{2}\right)\frac{\zeta(s)}{\sin\pi s}{}_1F_{1}\left(\frac{1-s}{2};\frac{1}{2};\frac{z^2}{4}\right)(\sqrt{\pi}\alpha)^{-s}\, ds\nonumber\\
&\quad-2\pi iL,
\end{align}
where
\begin{align}\label{LL}
L&:=\lim_{s\to 1}\frac{d}{ds}\left((s-1)^2\Gamma\left(\frac{s}{2}\right)\frac{\zeta(s)}{\sin\pi s}{}_1F_{1}\left(\frac{1-s}{2};\frac{1}{2};\frac{z^2}{4}\right)(\sqrt{\pi}\alpha)^{-s}\right)\nonumber\\
&=L_{1}+L_{2},
\end{align}
where
\begin{align*}
L_{1}&:=\lim_{s\to 1}\left\{\frac{d}{ds}\left((s-1)^2\Gamma\left(\frac{s}{2}\right)\frac{\zeta(s)}{\sin\pi s}(\sqrt{\pi}\alpha)^{-s}\right){}_1F_{1}\left(\frac{1-s}{2};\frac{1}{2};\frac{z^2}{4}\right)\right\},\nonumber\\
L_{2}&:=\lim_{s\to 1}\left\{(s-1)^2\Gamma\left(\frac{s}{2}\right)\frac{\zeta(s)}{\sin\pi s}(\sqrt{\pi}\alpha)^{-s}\frac{d}{ds}{}_1F_{1}\left(\frac{1-s}{2};\frac{1}{2};\frac{z^2}{4}\right)\right\}.
\end{align*}
Using (5.12) from \cite{series}, (\ref{xi}) and (\ref{rf}), we observe that
\begin{align}\label{dutk}
L_{1}&=\lim_{s\to 1}\frac{d}{ds}\left((s-1)^2\Gamma\left(\frac{s}{2}\right)\frac{\zeta(s)}{\sin\pi s}(\sqrt{\pi}\alpha)^{-s}\right)\nonumber\\
&=\frac{-2}{\pi}\lim_{s\to 1}\frac{d}{ds}\left((s-1)^2\Gamma(s-1)\Gamma(-s)\xi(s)\alpha^{-s}\right)\nonumber\\
&=(-\gamma+\log\left(4\pi\alpha^2\right))/(2\pi\alpha).
\end{align}
Now, 
\begin{align*}
\frac{d}{ds}{}_1F_{1}\left(\frac{1-s}{2};\frac{1}{2};\frac{z^2}{4}\right)&=\frac{d}{ds}\sum_{n=1}^{\infty}\frac{((1-s)/2)_{n}}{(1/2)_{n}}\frac{(z^2/4)^n}{n!}\nonumber\\
&=\sum_{n=1}^{\infty}\frac{(z^2/4)^n}{(1/2)_{n}n!}\frac{d}{ds}\prod_{j=0}^{n-1}\left(\frac{1-s}{2}+j\right)\nonumber\\
&=-\frac{1}{2}\sum_{n=1}^{\infty}\frac{(z^2/4)^n}{(1/2)_{n}n!}\left(\frac{1-s}{2}\right)_{n}\sum_{j=0}^{n-1}\frac{1}{\frac{1-s}{2}+j},
\end{align*}
so that
\begin{align}\label{hypcal1}
\lim_{s\to 1}\frac{d}{ds}{}_1F_{1}\left(\frac{1-s}{2};\frac{1}{2};\frac{z^2}{4}\right)&=-\frac{1}{2}\lim_{s\to 1}\sum_{n=1}^{\infty}\frac{(z^2/4)^n}{(1/2)_{n}n!}\left(\frac{3-s}{2}\right)_{n-1}\nonumber\\
&=-\frac{1}{2}\sum_{n=1}^{\infty}\frac{(z^2/4)^n}{(1/2)_{n}n}\nonumber\\
&=-\frac{z^2}{4}{}_2F_{2}(1,1;3/2,2;z^2/4).
\end{align}
Hence, $L_{2}=\frac{z^2}{4\pi\alpha}\cdot{}_2F_{2}(1,1;3/2,2;z^2/4)$ so that, by (\ref{LL}) and (\ref{dutk}),
\begin{equation}\label{vll}
L=\frac{1}{2\pi\alpha}\left(-\gamma+\log\left(4\pi\alpha^2\right)+\frac{z^2}{2}{}_2F_{2}(1,1;3/2,2;z^2/4)\right).
\end{equation}
As before, using (\ref{strivert}), one easily sees that the integrals on the horizontal segments $[\tf{1}{2}-iT,1+\delta-iT]$ and $[1+\delta+iT, \tf{1}{2}+iT]$ tend to zero as $T\to\infty$. Thus it remains to evaluate 
\begin{align}\label{hcalj0}
J(z,\alpha)&:=\int_{1+\delta-i\infty}^{1+\delta+i\infty}\Gamma\left(\frac{s}{2}\right)\frac{\zeta(s)}{\sin\pi s}{}_1F_{1}\left(\frac{1-s}{2};\frac{1}{2};\frac{z^2}{4}\right)(\sqrt{\pi}\alpha)^{-s}\, ds\nonumber\\
&=\sum_{n=1}^{\infty}\int_{1+\delta-i\infty}^{1+\delta+i\infty}\frac{\Gamma\left(\frac{s}{2}\right)}{\sin\pi s}{}_1F_{1}\left(\frac{1-s}{2};\frac{1}{2};\frac{z^2}{4}\right)(\sqrt{\pi}\alpha n)^{-s}\, ds\nonumber\\
&=:\sum_{n=1}^{\infty}J(z,\alpha,n),
\end{align}
where we have interchanged the order of summation and integration, which is valid because of absolute convergence. Another application of the residue theorem yields, for $0<c=$ Re $s<1$,
\begin{align}\label{resanh}
J(z,\alpha,n)&=\int_{c-i\infty}^{c+i\infty}\frac{\Gamma\left(\frac{s}{2}\right)}{\sin\pi s}{}_1F_{1}\left(\frac{1-s}{2};\frac{1}{2};\frac{z^2}{4}\right)(\sqrt{\pi}\alpha n)^{-s}\, ds\nonumber\\
&\quad\quad\quad+2\pi i\lim_{s\to 1}\frac{(s-1)\Gamma\left(\frac{s}{2}\right)}{\sin\pi s}{}_1F_{1}\left(\frac{1-s}{2};\frac{1}{2};\frac{z^2}{4}\right)(\sqrt{\pi}\alpha n)^{-s}.\nonumber\\
\end{align}
It is well-known that \cite[p.~91, Equation (3.3.10)]{kp}
\begin{equation}\label{osinmel}
\frac{1}{2\pi i}\int_{c-i\infty}^{c+i\infty}\frac{x^{-s}}{\sin\pi s}\, ds=\frac{1}{\pi (1+x)}.
\end{equation}
Also, from \cite[p.83, Equation (3.1.13)]{kp}, we have
\begin{equation}\label{melconv}
\frac{1}{2\pi i}\int_{c-i\infty}^{c+i\infty}F(s)G(s)w^{-s}\, ds=\int_{0}^{\infty}f(x)g\left(\frac{w}{x}\right)\frac{dx}{x},
\end{equation}
where $F(s)$ and $G(s)$ are Mellin transforms of $f(x)$ and $g(x)$ respectively.
Hence, from (\ref{invmel1}), (\ref{osinmel}) and (\ref{melconv}), we have
\begin{align}\label{aar1}
\int_{c-i\infty}^{c+i\infty}\frac{\Gamma\left(\frac{s}{2}\right)}{\sin\pi s}{}_1F_{1}\left(\frac{1-s}{2};\frac{1}{2};\frac{z^2}{4}\right)(\sqrt{\pi}\alpha n)^{-s}\, ds=4ie^{z^{2}/4}\int_{0}^{\infty}\frac{e^{-x^2}\cos xz}{x+\sqrt{\pi}n\alpha}\, dx.
\end{align}
Also,
\begin{equation}\label{aal1}
\lim_{s\to 1}\frac{(s-1)\Gamma\left(\frac{s}{2}\right)}{\sin\pi s}{}_1F_{1}\left(\frac{1-s}{2};\frac{1}{2};\frac{z^2}{4}\right)(\sqrt{\pi}\alpha n)^{-s}=-\frac{1}{\pi n\alpha}.
\end{equation}
Thus, (\ref{resanh}), (\ref{aar1}) and (\ref{aal1}) imply that
\begin{align}\label{resanh1}
J(z,n,\alpha)=4ie^{z^{2}/4}\left(\int_{0}^{\infty}\frac{e^{-x^2}\cos xz}{x+\sqrt{\pi}n\alpha}\, dx-\frac{e^{-z^2}/4}{2n\alpha}\right).
\end{align}
Rewriting $e^{-z^2/4}$ as an integral, we have
\begin{equation}\label{inrep}
e^{-z^2/4}=\frac{2}{\sqrt{\pi}}\int_{0}^{\infty}e^{-x^2}\cos xz\, dx.
\end{equation}
Now (\ref{resanh1}) along with (\ref{inrep}), (\ref{hcalj0}) and (\ref{w1.15b})
give
\begin{align}\label{vlli}
J(z,\alpha)&=4ie^{z^{2}/4}\sum_{n=1}^{\infty}\int_{0}^{\infty}e^{-x^2}\cos(xz)\left\{\frac{1}{x+\sqrt{\pi} n\alpha}-\frac{1}{\sqrt{\pi}n\alpha}\right\}\, dx\nonumber\\
&=-4ie^{z^{2}/4}\sum_{n=1}^{\infty}\int_{0}^{\infty}e^{-\pi\alpha^2 x^2}\cos(\sqrt{\pi}\alpha xz)\left\{\frac{1}{n}-\frac{1}{x+n}\right\}\, dx\nonumber\\
&=-4ie^{z^{2}/4}\int_{0}^{\infty}e^{-\pi\alpha^2 x^2}\cos(\sqrt{\pi}\alpha xz)\sum_{n=0}^{\infty}\left\{\frac{1}{n+1}-\frac{1}{x+1+n}\right\}\, dx\nonumber\\
&=-4ie^{z^{2}/4}\int_{0}^{\infty}\left(\psi(x+1)+\gamma\right)e^{-\pi\alpha^2 x^2}\cos(\sqrt{\pi}\alpha xz)\, dx\nonumber\\
&=-4ie^{z^{2}/4}\left(\frac{\gamma e^{-z^2/4}}{2\alpha}+\int_{0}^{\infty}\psi(x+1)e^{-\pi\alpha^2 x^2}\cos(\sqrt{\pi}\alpha xz)\, dx\right),
\end{align}%
where in the second step we made the change of variable $x\to\sqrt{\pi}\alpha x$ and in the third step, we interchanged the order of summation and integration, which is valid because of absolute convergence. Thus from (\ref{nsphar}), (\ref{resimpact2}), (\ref{vll}) and (\ref{vlli}), we have
\begin{align}\label{beffis}
&\int_{0}^{\infty}\frac{\Xi(\tf{1}{2}t)}{1+t^2}\frac{\nabla\left(\alpha, z, \frac{1+it}{2}\right)}{\cosh \tf{1}{2}\pi t}\, dt\nonumber\\
&=-\frac{\alpha^{\frac{1}{2}}e^{-\frac{z^2}{8}}}{4i}\bigg\{-4ie^{z^{2}/4}\left(\frac{\gamma e^{-z^2/4}}{2\alpha}+\int_{0}^{\infty}\psi(x+1)e^{-\pi\alpha^2 x^2}\cos(\sqrt{\pi}\alpha xz)\, dx\right)\nonumber\\
&\quad\quad\quad\quad\quad\quad-\frac{i}{\alpha}\left(-\gamma+\log\left(4\pi\alpha^2\right)+\frac{z^2}{2}{}_2F_{2}(1,1;3/2,2;z^2/4)\right)\bigg\}\nonumber\\
&=\sqrt{\alpha}e^{\frac{z^2}{8}}\int_{0}^{\infty}\psi(x+1)e^{-\pi\alpha^2 x^2}\cos(\sqrt{\pi}\alpha xz)\, dx\nonumber\\
&+\quad\frac{e^{-\frac{z^2}{8}}}{4\sqrt{\alpha}}\left(\gamma+\log\left(4\pi\alpha^2\right)+\frac{z^2}{2}{}_2F_{2}(1,1;3/2,2;z^2/4)\right).
\end{align}
Finally, replacing $z$ by $iz$ and $\alpha$ by $\beta$ in (\ref{beffis}), noting that the integral on the left-hand side remains invariant in this process, and then combining the result with (\ref{beffis}), we arrive at 
{\allowdisplaybreaks\begin{align}\label{beffis1}
&\int_{0}^{\infty}\frac{\Xi(\tf{1}{2}t)}{1+t^2}\frac{\nabla\left(\alpha, z, \frac{1+it}{2}\right)}{\cosh \tf{1}{2}\pi t}\, dt\nonumber\\
&=\sqrt{\alpha}e^{\frac{z^2}{8}}\int_{0}^{\infty}\psi(x+1)e^{-\pi\alpha^2 x^2}\cos(\sqrt{\pi}\alpha xz)\, dx\nonumber\\
&+\quad\frac{e^{-\frac{z^2}{8}}}{4\sqrt{\alpha}}\left(\gamma+\log\left(4\pi\alpha^2\right)+\frac{z^2}{2}{}_2F_{2}(1,1;3/2,2;z^2/4)\right)\nonumber\\
&=\sqrt{\beta}e^{-\frac{z^2}{8}}\int_{0}^{\infty}\psi(x+1)e^{-\pi\beta^2 x^2}\cosh(\sqrt{\pi}\beta xz)\, dx\nonumber\\
&+\quad\frac{e^{\frac{z^2}{8}}}{4\sqrt{\beta}}\left(\gamma+\log\left(4\pi\beta^2\right)-\frac{z^2}{2}{}_2F_{2}(1,1;3/2,2;-z^2/4)\right).
\end{align}}%
We can rewrite (\ref{beffis1}) in a more compact form by means of the integral evaluation
\begin{equation}\label{intimp}
\int_{0}^{\infty}e^{-\pi\alpha^2 x^2}\cos(\sqrt{\pi}\alpha xz)\log x\, dx=-\frac{e^{-\frac{z^2}{4}}}{4\alpha}\left(\gamma+\log\left(4\pi\alpha^2\right)+\frac{z^2}{2}{}_2F_{2}(1,1;3/2,2;z^2/4)\right),
\end{equation}
which can be proved by expanding $\cos(\sqrt{\pi}\alpha xz)$ into infinite series, interchanging the order of summation and integration and then employing the following formula \cite[p.573, formula 4.352, no. 3]{grn}, valid for Re $\mu>0$,
\begin{equation*}
\int_{0}^{\infty}x^{n-\frac{1}{2}}e^{-\mu x}\log x\, dx=\sqrt{\pi}\frac{(2n-1)!!}{2^n\mu^{n+\frac{1}{2}}}\left[2\left(1+\frac{1}{3}+\cdots+\frac{1}{2n-1}\right)-\gamma-\log 4\mu\right]
\end{equation*}
along with the fact that
\begin{equation*}
\sum_{n=0}^{\infty}\frac{(-z^2/4)^n}{n!}\left(1+\frac{1}{3}+\cdots\frac{1}{2n-1}\right)=-\frac{z^2}{4}e^{-z^2/4}{}_2F_{2}(1,1;3/2,2;z^2/4),
\end{equation*}
which in turn can be proved by reversing the steps in (\ref{hypcal1}) and using (\ref{kft}). Combining (\ref{intimp}) with (\ref{beffis1}), we obtain (\ref{hfg}).
\section{Generalization of Ferrar's formula}
Using (\ref{xif}), (\ref{xi}), (\ref{nabla}), (\ref{rho}), (\ref{strivert}) and (\ref{conasy}), we easily see that the integral on the extreme right-hand side of (\ref{gf}) converges. Let $\phi(s)=\frac{\sqrt{2}}{\frac{1}{2}-s}\Gamma\left(\frac{1}{4}+\frac{s}{2}\right)$ so that $f(\frac{t}{2})=\phi\left(\frac{it}{2}\right)\phi\left(\frac{-it}{2}\right)=\frac{8}{1+t^{2}}\Gamma\left(\frac{1+it}{4}\right)\Gamma\left(\frac{1-it}{4}\right)$. Substituting this in (\ref{sp4}) and using (\ref{xi}) and (\ref{rho}), we have
\begin{align}\label{nsp4.5}
&8\int_{0}^{\infty}\Gamma\left(\frac{1+it}{4}\right)\Gamma\left(\frac{1-it}{4}\right)\frac{\Xi(t/2)}{1+t^2}\nabla\left(\alpha,z,\frac{1+it}{2}\right)\, dt\nonumber\\
&=2i\int_{\frac{1}{2}-i\infty}^{\frac{1}{2}+i\infty}\pi^{-\frac{s}{2}}\Gamma^{2}\left(\frac{s}{2}\right)\Gamma\left(\frac{1-s}{2}\right)\zeta(s)\rho(\alpha,z,s)\, ds\nonumber\\
&=2i\alpha^{\frac{1}{2}}e^{-\frac{z^2}{8}}\int_{\frac{1}{2}-i\infty}^{\frac{1}{2}+i\infty}\Gamma^{2}\left(\frac{s}{2}\right)\Gamma\left(\frac{1-s}{2}\right)\zeta(s){}_1F_{1}\left(\frac{1-s}{2};\frac{1}{2};\frac{z^2}{4}\right)(\sqrt{\pi}\alpha)^{-s}\, ds,
\end{align}
To evaluate the last integral, we wish to shift the line of integration from Re $s=1/2$ to Re $s=1+\delta$, $0<\delta<2$, so that we can use (\ref{zzdefgr}). Consider a positively oriented rectangular contour with sides $[\frac{1}{2}+iT, \frac{1}{2}-iT], [\frac{1}{2}-iT, 1+\delta-iT], [1+\delta-iT,1+\delta+iT]$ and $[1+\delta+iT,\frac{1}{2}+iT]$, where $T$ is any positive real number. While shifting the line of integration, we encounter the pole of order two at $s=1$ (due to $\Gamma\left(\frac{1-s}{2}\right)$ and $\zeta(s)$). Using the residue theorem, we have
\begin{align}\label{resimpact1}
&\int_{\frac{1}{2}-iT}^{\frac{1}{2}+iT}\Gamma^{2}\left(\frac{s}{2}\right)\Gamma\left(\frac{1-s}{2}\right)\zeta(s){}_1F_{1}\left(\frac{1-s}{2};\frac{1}{2};\frac{z^2}{4}\right)(\sqrt{\pi}\alpha)^{-s}\, ds\nonumber\\
&=\left[\int_{\tf{1}{2}-iT}^{1+\delta-iT}+\int_{1+\delta-iT}^{1+\delta+iT}+\int_{1+\delta+iT}^{\tf{1}{2}+iT}\right]\Gamma^{2}\left(\frac{s}{2}\right)\Gamma\left(\frac{1-s}{2}\right)\zeta(s){}_1F_{1}\left(\frac{1-s}{2};\frac{1}{2};\frac{z^2}{4}\right)(\sqrt{\pi}\alpha)^{-s}\, ds\nonumber\\
&\quad\quad-2\pi iL,
\end{align}
where
\begin{align}\label{L}
L&:=\lim_{s\to 1}\frac{d}{ds}\left((s-1)^2\Gamma^{2}\left(\frac{s}{2}\right)\Gamma\left(\frac{1-s}{2}\right)\zeta(s){}_1F_{1}\left(\frac{1-s}{2};\frac{1}{2};\frac{z^2}{4}\right)(\sqrt{\pi}\alpha)^{-s}\right)\nonumber\\
&=L_{1}+L_{2},
\end{align}
with
\begin{align*}
L_{1}&:=\lim_{s\to 1}\left\{\frac{d}{ds}\left((s-1)^2\Gamma^{2}\left(\frac{s}{2}\right)\Gamma\left(\frac{1-s}{2}\right)\zeta(s)(\sqrt{\pi}\alpha)^{-s}\right){}_1F_{1}\left(\frac{1-s}{2};\frac{1}{2};\frac{z^2}{4}\right)\right\}\nonumber\\
L_{2}&:=\lim_{s\to 1}\left\{(s-1)^2\Gamma^{2}\left(\frac{s}{2}\right)\Gamma\left(\frac{1-s}{2}\right)\zeta(s)(\sqrt{\pi}\alpha)^{-s}\frac{d}{ds}{}_1F_{1}\left(\frac{1-s}{2};\frac{1}{2};\frac{z^2}{4}\right)\right\}.
\end{align*}
Now $L_{1}$ can be easily computed, or from \cite[Equation (4.12)]{series}, we readily have
\begin{align}\label{compul1}
L_{1}&=\lim_{s\to 1}\frac{d}{ds}\left((s-1)^2\Gamma^{2}\left(\frac{s}{2}\right)\Gamma\left(\frac{1-s}{2}\right)\zeta(s)\left(\sqrt{\pi}\alpha\right)^{-s}\right)\nonumber\\
&=\frac{\sqrt{\pi}}{\alpha}\left(\log 16\pi+2\log\alpha-\gamma\right).
\end{align}
Also, from (\ref{hypcal1}) and the fact that $\Gamma(1/2)=\sqrt{\pi}$, we find that
\begin{align}\label{compul2}
L_{2}&=-\frac{z^2}{4}{}_2F_{2}(1,1;3/2,2;z^2/4)\lim_{s\to 1}\left\{(s-1)^2\Gamma^{2}\left(\frac{s}{2}\right)\Gamma\left(\frac{1-s}{2}\right)\zeta(s)(\sqrt{\pi}\alpha)^{-s}\right\}\nonumber\\
&=\frac{z^2\sqrt{\pi}}{2\alpha}{}_2F_{2}(1,1;3/2,2;z^2/4).
\end{align}
Finally, from (\ref{L}), (\ref{compul1}) and (\ref{compul2}), we have
\begin{equation}\label{fcoml}
L=\frac{\sqrt{\pi}}{\alpha}\left(-\gamma+\log 16\pi+2\log\alpha+\frac{z^2}{2}{}_2F_{2}(1,1;3/2,2;z^2/4)\right).
\end{equation}
Using (\ref{strivert}), one easily sees that the integrals on the horizontal segments $[\tf{1}{2}-iT,1+\delta-iT]$ and $[1+\delta+iT, \tf{1}{2}+iT]$ tend to zero as $T\to\infty$. Thus it remains to evaluate 
\begin{align}\label{calj0}
J(z,\alpha)&:=\int_{1+\delta-i\infty}^{1+\delta+i\infty}\Gamma^{2}\left(\frac{s}{2}\right)\Gamma\left(\frac{1-s}{2}\right)\zeta(s){}_1F_{1}\left(\frac{1-s}{2};\frac{1}{2};\frac{z^2}{4}\right)(\sqrt{\pi}\alpha)^{-s}\, ds\nonumber\\
&=\sqrt{\pi}\sum_{n=1}^{\infty}\int_{1+\delta-i\infty}^{1+\delta+i\infty}B\left(\frac{s}{2},\frac{1-s}{2}\right)\Gamma\left(\frac{s}{2}\right){}_1F_{1}\left(\frac{1-s}{2};\frac{1}{2};\frac{z^2}{4}\right)(\sqrt{\pi}\alpha n)^{-s}\, ds\nonumber\\
&=:\sum_{n=1}^{\infty}J(z,\alpha,n).
\end{align}
Here $B(s,z-s)$ is the Euler beta function given by
\begin{equation}\label{betamel}
B(s,z-s)=\int_{0}^{\infty}\frac{x^{s-1}}{(1+x)^{z}}\, dx=\frac{\Gamma(s)\Gamma(z-s)}{\Gamma(z)},\hspace{2mm}0<\hspace{1mm}\text{Re}\hspace{1mm}s<\hspace{1mm}\text{Re}\hspace{1mm}z.
\end{equation}
Another application of the residue theorem yields, for $0<c=$ Re $s<1$,
\begin{align}\label{resan}
J(z,\alpha,n)&=\sqrt{\pi}\bigg(\int_{c-i\infty}^{c+i\infty}B\left(\frac{s}{2},\frac{1-s}{2}\right)\Gamma\left(\frac{s}{2}\right){}_1F_{1}\left(\frac{1-s}{2};\frac{1}{2};\frac{z^2}{4}\right)(\sqrt{\pi}\alpha n)^{-s}\, ds\nonumber\\
&\quad\quad\quad+2\pi i\lim_{s\to 1}\frac{(s-1)}{\sqrt{\pi}}\Gamma^{2}\left(\frac{s}{2}\right)\Gamma\left(\frac{1-s}{2}\right){}_1F_{1}\left(\frac{1-s}{2};\frac{1}{2};\frac{z^2}{4}\right)(\sqrt{\pi}\alpha n)^{-s}\bigg)\nonumber\\
&=\sqrt{\pi}\bigg(\int_{c-i\infty}^{c+i\infty}B\left(\frac{s}{2},\frac{1-s}{2}\right)\Gamma\left(\frac{s}{2}\right){}_1F_{1}\left(\frac{1-s}{2};\frac{1}{2};\frac{z^2}{4}\right)(\sqrt{\pi}\alpha n)^{-s}\, ds-\frac{4\pi i}{n\alpha}\bigg).
\end{align}
From (\ref{betamel}), we have for $0<c=$ Re $s<1$,
\begin{equation}\label{betamel1}
\frac{1}{2\pi i}\int_{c-i\infty}^{c+i\infty}B\left(\frac{s}{2},\frac{1-s}{2}\right)x^{-s}\, ds=\frac{2}{\sqrt{1+x^2}}.
\end{equation}
Now using (\ref{invmel1}), (\ref{betamel1}) and (\ref{melconv}), we see that
\begin{align}\label{calj}
\int_{c-i\infty}^{c+i\infty}B\left(\frac{s}{2},\frac{1-s}{2}\right)\Gamma\left(\frac{s}{2}\right){}_1F_{1}\left(\frac{1-s}{2};\frac{1}{2};\frac{z^2}{4}\right)(\sqrt{\pi}\alpha n)^{-s}\, ds=8\pi ie^{z^2/4}\int_{0}^{\infty}\frac{e^{-x^2}\cos xz}{\sqrt{x^2+\pi \alpha^2n^2}}\, dx.
\end{align}
Hence, from (\ref{calj0}), (\ref{resan}) and (\ref{calj}), we deduce that
\begin{equation}\label{calj1}
J(z,\alpha)=8\pi^{3/2}ie^{z^2/4}\sum_{n=1}^{\infty}\left(\int_{0}^{\infty}\frac{e^{-x^2}\cos xz}{\sqrt{x^2+\pi \alpha^2n^2}}\, dx-\frac{e^{-z^2/4}}{2n\alpha}\right).
\end{equation}
Substituting (\ref{inrep}) in (\ref{calj1}) and then employing the change of variable $x=\alpha t/(2\sqrt{\pi})$, we have
\begin{align}\label{j1rep}
J(z,\alpha)&=8\pi^{3/2}ie^{z^2/4}\sum_{n=1}^{\infty}\int_{0}^{\infty}e^{-x^2}\cos xz\left(\frac{1}{\sqrt{x^2+\pi \alpha^2n^2}}-\frac{1}{\sqrt{\pi}n\alpha}\right)\, dx\nonumber\\
&=8\pi^{3/2}ie^{z^2/4}\sum_{n=1}^{\infty}\int_{0}^{\infty}e^{-\frac{\alpha^2t^2}{4\pi}}\cos\left(\frac{\alpha tz}{2\sqrt{\pi}}\right)\left(\frac{1}{\sqrt{t^2+4\pi^2n^2}}-\frac{1}{2\pi n}\right)\, dt\nonumber\\
&=8\pi^{3/2}ie^{z^2/4}\int_{0}^{\infty}e^{-\frac{\alpha^2t^2}{4\pi}}\cos\left(\frac{\alpha tz}{2\sqrt{\pi}}\right)\sum_{n=1}^{\infty}\left(\frac{1}{\sqrt{t^2+4\pi^2n^2}}-\frac{1}{2\pi n}\right)\, dt.
\end{align}
Now from \cite[Equation 6]{wats}, we have, for Re $z>0$,
\begin{equation}\label{wat}
2\sum_{n=1}^{\infty}K_{0}(nz)=\pi\left\{\frac{1}{z}+2\sum_{n=1}^{\infty}\left(\frac{1}{\sqrt{z^2+4\pi^2n^2}}-\frac{1}{2n\pi}\right)\right\}+\gamma+\log\left(\frac{z}{2}\right)-\log 2\pi.
\end{equation}
From (\ref{j1rep}) and (\ref{wat}), we have
\begin{align}\label{jaftw}
J(z,\alpha)&=8\pi^{3/2}ie^{z^2/4}\int_{0}^{\infty}e^{-\frac{\alpha^2t^2}{4\pi}}\cos\left(\frac{\alpha tz}{2\sqrt{\pi}}\right)\left(\frac{1}{2\pi}\left(-\gamma+\log 4\pi-\log t+2\sum_{n=1}^{\infty}K_{0}(nt)\right)-\frac{1}{2t}\right)\, dt\nonumber\\
&=:8\pi^{3/2}ie^{z^2/4}\left(J_{1}(z,\alpha)+J_{2}(z,\alpha)+J_{3}(z,\alpha)\right), 
\end{align}
where
\begin{align*}
J_{1}(z,\alpha)&:=\frac{(-\gamma+\log 4\pi)}{2\pi}\int_{0}^{\infty}e^{-\frac{\alpha^2t^2}{4\pi}}\cos\left(\frac{\alpha tz}{2\sqrt{\pi}}\right)\, dt,\nonumber\\
J_{2}(z,\alpha)&:=-\frac{1}{2\pi}\int_{0}^{\infty}e^{-\frac{\alpha^2t^2}{4\pi}}\cos\left(\frac{\alpha tz}{2\sqrt{\pi}}\right)\log t\, dt,\nonumber\\
J_{3}(z,\alpha)&:=\frac{1}{\pi}\int_{0}^{\infty}e^{-\frac{\alpha^2t^2}{4\pi}}\cos\left(\frac{\alpha tz}{2\sqrt{\pi}}\right)\left(\sum_{n=1}^{\infty}K_{0}(nt)-\frac{\pi}{2t}\right)\, dt.\nonumber\\
\end{align*}
However, from \cite[p.~488, formula 3.896, no. 4]{grn}
\begin{equation}\label{j1e}
J_{1}(z,\alpha)=\frac{e^{-z^2/4}}{2\alpha}(-\gamma+\log 4\pi)
\end{equation}
and by using (\ref{intimp}), we find that
\begin{equation}\label{j2e}
J_{2}(z,\alpha)=\frac{e^{-z^2/4}}{4\alpha}\left(\gamma+2\log\alpha-\log\pi+\frac{z^2}{2}{}_2F_{2}(1,1;3/2,2;z^2/4)\right).
\end{equation}
Thus, from (\ref{nsp4.5}), (\ref{resimpact1}), (\ref{fcoml}), (\ref{jaftw}), (\ref{j1e}) and (\ref{j2e}), we deduce that
\begin{align*}
&8\int_{0}^{\infty}\Gamma\left(\frac{1+it}{4}\right)\Gamma\left(\frac{1-it}{4}\right)\frac{\Xi(t/2)}{1+t^2}\nabla\left(\alpha,z,\frac{1+it}{2}\right)\, dt\nonumber\\
&=-4\pi^{3/2}\alpha^{\frac{1}{2}}e^{-\frac{z^2}{8}}\bigg\{\frac{2}{\alpha}(-\gamma+\log 4\pi)+\frac{1}{\alpha}\left(\gamma+2\log\alpha-\log\pi+\frac{z^2}{2}{}_2F_{2}(1,1;3/2,2;z^2/4)\right)\nonumber\\
&\quad\quad\quad\quad\quad\quad\quad\quad+\frac{4e^{z^2/4}}{\pi}\int_{0}^{\infty}e^{-\frac{\alpha^2t^2}{4\pi}}\cos\left(\frac{\alpha tz}{2\sqrt{\pi}}\right)\left(\sum_{n=1}^{\infty}K_{0}(nt)-\frac{\pi}{2t}\right)\, dt\nonumber\\
&\quad\quad\quad\quad\quad\quad\quad\quad-\frac{1}{\alpha}\left(-\gamma+\log 16\pi+2\log\alpha+\frac{z^2}{2}{}_2F_{2}(1,1;3/2,2;z^2/4)\right)\bigg\}\nonumber\\
&=-16\sqrt{\pi}\sqrt{\alpha}e^{\frac{z^2}{8}}\int_{0}^{\infty}e^{-\frac{\alpha^2t^2}{4\pi}}\cos\left(\frac{\alpha tz}{2\sqrt{\pi}}\right)\left(\sum_{n=1}^{\infty}K_{0}(nt)-\frac{\pi}{2t}\right)\, dt,
\end{align*}
so that
\begin{align}\label{evfm}
&-\frac{1}{2\sqrt{\pi}}\int_{0}^{\infty}\Gamma\left(\frac{1+it}{4}\right)\Gamma\left(\frac{1-it}{4}\right)\frac{\Xi(t/2)}{1+t^2}\nabla\left(\alpha,z,\frac{1+it}{2}\right)\, dt.\nonumber\\
&=\sqrt{\alpha}e^{\frac{z^2}{8}}\int_{0}^{\infty}e^{-\frac{\alpha^2t^2}{4\pi}}\cos\left(\frac{\alpha tz}{2\sqrt{\pi}}\right)\left(\sum_{n=1}^{\infty}K_{0}(nt)-\frac{\pi}{2t}\right)\, dt.\nonumber\\
\end{align}
Finally, replacing $z$ by $iz$ and $\alpha$ by $\beta$ in (\ref{evfm}), noting that the integral on the left-hand side remains invariant under this replacement, and then combining the result with (\ref{evfm}), we arrive at (\ref{gf}).\\

\textbf{Remark.} The special case (\ref{ferex}) of (\ref{gf}) can be derived as follows. Let $z=0$ in (\ref{gf}). Then,
\begin{align}\label{gf0}
\sqrt{\alpha}\int_{0}^{\infty}e^{-\frac{\alpha^2t^2}{4\pi}}\left(\sum_{n=1}^{\infty}K_{0}(nt)-\frac{\pi}{2t}\right)\, dt
&=\sqrt{\beta}\int_{0}^{\infty}e^{-\frac{\beta^2t^2}{4\pi}}\left(\sum_{n=1}^{\infty}K_{0}(nt)-\frac{\pi}{2t}\right)\, dt\nonumber\\
&=-\frac{1}{\sqrt{\pi}}\int_{0}^{\infty}\Gamma\left(\frac{1+it}{4}\right)\Gamma\left(\frac{1-it}{4}\right)\Xi\left(\frac{t}{2}\right)\frac{\cos\left(\frac{1}{2}t\log\alpha\right)}{1+t^2}\, dt.\nonumber\\
\end{align}
From the invariance of the integral under the map $\alpha\to\beta$, it suffices to show the equality of the extreme left and right expressions in (\ref{ferex}). To that end, observe that using (\ref{wat}) in the extreme left and right sides of (\ref{gf0}) and using (\ref{j1e}) and (\ref{j2e}), we have
{\allowdisplaybreaks\begin{align}\label{ferosolv}
&4\pi^{-\frac{3}{2}}\int_{0}^{\infty}\Gamma\left(\frac{1+it}{4}\right)\Gamma\left(\frac{1-it}{4}\right)\Xi\left(\frac{t}{2}\right)\frac{\cos\left(\frac{1}{2}t\log\alpha\right)}{1+t^2}\, dt\nonumber\\
&=-4\sqrt{\alpha}\int_{0}^{\infty}e^{-\frac{\alpha^2t^2}{4\pi}}\left(\frac{\gamma-\log 4\pi}{2\pi}+\frac{\log t}{2\pi}+\sum_{n=1}^{\infty}\left(\frac{1}{\sqrt{t^2+4\pi^2n^2}}-\frac{1}{2n\pi}\right)\right)\, dt\nonumber\\
&=4\sqrt{\alpha}\left(J_{1}(0,\alpha)+J_{2}(0,\alpha)-\int_{0}^{\infty}e^{-\frac{\alpha^2t^2}{4\pi}}\sum_{n=1}^{\infty}\left(\frac{1}{\sqrt{t^2+4\pi^2n^2}}-\frac{1}{2n\pi}\right)\, dt\right)\nonumber\\
&=4\sqrt{\alpha}\left(\frac{-\gamma+\log 4\pi}{2\alpha}+\frac{1}{4\alpha}\left(\gamma+2\log\alpha-\log\pi\right)-\sum_{n=1}^{\infty}\int_{0}^{\infty}e^{-\frac{\alpha^2t^2}{4\pi}}\left(\frac{1}{\sqrt{t^2+4\pi^2n^2}}-\frac{1}{2n\pi}\right)\, dt\right)\nonumber\\
&=4\sqrt{\alpha}\left(-\frac{\gamma}{4\alpha}+\frac{\log\alpha}{2\alpha}+\frac{\log 16\pi}{4\alpha}-\sum_{n=1}^{\infty}\left(\int_{0}^{\infty}\frac{e^{-\frac{\alpha^2t^2}{4\pi}}\, dt}{\sqrt{t^2+4\pi^2n^2}}-\frac{1}{2n\alpha}\right)\right).
\end{align}}%
Employing a change of variable $x=t^2$ in the formula \cite[p.~351, formula 3.388, no. 2]{grn}
\begin{equation*}
\int_{0}^{\infty}(2bx+x^2)^{\nu-1}e^{-\mu x}\, dx=\frac{1}{\sqrt{\pi}}\left(\frac{2b}{\mu}\right)^{\nu-1/2}e^{bu}\Gamma(\nu)K_{\nu-\frac{1}{2}}(b\mu),
\end{equation*}
valid for $|\arg b|<\pi$, Re $\nu>0$, Re $\mu>0$ and then letting $\nu=\tfrac{1}{2}$, $\mu=\frac{\alpha^2}{4\pi}$ and $b=2\pi^2n^2$, we observe that
\begin{equation}\label{fersolv}
\int_{0}^{\infty}\frac{e^{-\frac{\alpha^2t^2}{4\pi}}\, dt}{\sqrt{t^2+4\pi^2n^2}}=\frac{1}{2}e^{\frac{\pi\alpha^{2}n^2}{2}}K_{0}\left(\frac{\pi\alpha^{2}n^2}{2}\right).
\end{equation}
Substituting (\ref{fersolv}) in (\ref{ferosolv}) and simplifying, we arrive at
\begin{align}\label{ferex1}
&\sqrt{\alpha}\left(\frac{-\gamma+\log 16\pi+2\log\alpha}{\alpha}-2\sum_{n=1}^{\infty}\left(e^{\frac{\pi \alpha^2n^2}{2}}K_{0}\left(\frac{\pi \alpha^2n^2}{2}\right)-\frac{1}{n\alpha}\right)\right)\nonumber\\
&=4\pi^{-\frac{3}{2}}\int_{0}^{\infty}\Gamma\left(\frac{1+it}{4}\right)\Gamma\left(\frac{1-it}{4}\right)\Xi\left(\frac{t}{2}\right)\frac{\cos\left(\frac{1}{2}t\log\alpha\right)}{1+t^2}\, dt.
\end{align}
Now replace $\alpha$ by $\beta$ in (\ref{ferex1}) and note the invariance of the integral on the right-hand side under this transformation. This gives (\ref{ferex}).

Since the above steps are reversible, (\ref{gf0}) is equivalent to (\ref{ferex}).
\section{Generalization of a formula of Ramanujan}
Let 
\begin{equation*}
K(z,\alpha):=\int_{-\infty}^{\infty}\Gamma\left(\frac{-1+it}{4}\right)\Gamma\left(\frac{-1-it}{4}\right)\Xi\left(\frac{t}{2}\right)\rho\left(\alpha,z,\frac{3+it}{2}\right)\, dt,
\end{equation*}
where $\rho$ is defined in (\ref{rho}). Using (\ref{xif}), (\ref{xi}), (\ref{rho}), (\ref{strivert}) and (\ref{conasy}), we see that $K(z,\alpha)$ converges. Converting $K(z,\alpha)$ into a complex integral by the change of variable $s=\tfrac{1+it}{2}$ and employing (\ref{rho}) and (\ref{xi}), we observe that
{\allowdisplaybreaks\begin{align}\label{kzac}
K(z,\alpha)&=\frac{2}{i}\int_{\frac{1}{2}-i\infty}^{\frac{1}{2}+i\infty}\Gamma\left(\frac{s-1}{2}\right)\Gamma\left(-\frac{s}{2}\right)\xi(s)\rho\left(\alpha,z,s+1\right)\, ds\nonumber\\
&=\frac{-4\alpha^{-\frac{1}{2}}e^{-\frac{z^2}{8}}}{i}\int_{\frac{1}{2}-i\infty}^{\frac{1}{2}+i\infty}\Gamma\left(\frac{s+1}{2}\right)\Gamma\left(1-\frac{s}{2}\right){}_1F_{1}\left(-\frac{s}{2},\frac{1}{2};\frac{z^2}{4}\right)\Gamma\left(\frac{s}{2}\right)\zeta(s)(\sqrt{\pi}\alpha)^{-s}\, ds\nonumber\\
&=4i\alpha^{-\frac{1}{2}}e^{-\frac{z^2}{8}}\int_{\frac{1}{2}-i\infty}^{\frac{1}{2}+i\infty}\frac{\pi}{\sin\tfrac{1}{2}\pi s}\Gamma\left(\frac{s+1}{2}\right){}_1F_{1}\left(-\frac{s}{2},\frac{1}{2};\frac{z^2}{4}\right)\zeta(s)(\sqrt{\pi}\alpha)^{-s}\, ds,
\end{align}}%
where we used (\ref{rf}) in the last step. Let $T>0$ be a real number. Using the residue theorem, we have
\begin{align}\label{resimpactr}
&\int_{\frac{1}{2}-iT}^{\frac{1}{2}+iT}\frac{\pi}{\sin\tfrac{1}{2}\pi s}\Gamma\left(\frac{s+1}{2}\right){}_1F_{1}\left(-\frac{s}{2},\frac{1}{2};\frac{z^2}{4}\right)\zeta(s)(\sqrt{\pi}\alpha)^{-s}\, ds\nonumber\\
&=\left[\int_{\tf{1}{2}-iT}^{1+\delta-iT}+\int_{1+\delta-iT}^{1+\delta+iT}+\int_{1+\delta+iT}^{\tf{1}{2}+iT}\right]\frac{\pi}{\sin\tfrac{1}{2}\pi s}\Gamma\left(\frac{s+1}{2}\right){}_1F_{1}\left(-\frac{s}{2},\frac{1}{2};\frac{z^2}{4}\right)\zeta(s)(\sqrt{\pi}\alpha)^{-s}\, ds\nonumber\\
&\quad-2\pi iL,
\end{align}
where
\begin{align}\label{lrc}
L&:=\lim_{s\to 1}(s-1)\zeta(s)\frac{\pi}{\sin\tfrac{1}{2}\pi s}\Gamma\left(\frac{s+1}{2}\right){}_1F_{1}\left(-\frac{s}{2},\frac{1}{2};\frac{z^2}{4}\right)(\sqrt{\pi}\alpha)^{-s}\nonumber\\
&=\frac{\sqrt{\pi}}{\alpha}{}_1F_{1}\left(-\frac{1}{2},\frac{1}{2};\frac{z^2}{4}\right).
\end{align}
As before, using (\ref{strivert}), one easily sees that the integrals on the horizontal segments $[\tf{1}{2}-iT,1+\delta-iT]$ and $[1+\delta+iT, \tf{1}{2}+iT]$ tend to zero as $T\to\infty$. Thus it remains to evaluate
\begin{align}\label{calj0r}
J(z,\alpha)&:=\int_{1+\delta-i\infty}^{1+\delta+i\infty}\frac{\pi}{\sin\tfrac{1}{2}\pi s}\Gamma\left(\frac{s+1}{2}\right){}_1F_{1}\left(-\frac{s}{2},\frac{1}{2};\frac{z^2}{4}\right)\zeta(s)(\sqrt{\pi}\alpha)^{-s}\, ds\nonumber\\
&=\sum_{n=1}^{\infty}\int_{1+\delta-i\infty}^{1+\delta+i\infty}\frac{\pi}{\sin\tfrac{1}{2}\pi s}\Gamma\left(\frac{s+1}{2}\right){}_1F_{1}\left(-\frac{s}{2},\frac{1}{2};\frac{z^2}{4}\right)(\sqrt{\pi}\alpha n)^{-s}\, ds\nonumber\\
&=\sum_{n=1}^{\infty}\int_{c-i\infty}^{c+i\infty}\frac{\pi}{\sin\tfrac{1}{2}\pi s}\Gamma\left(\frac{s+1}{2}\right){}_1F_{1}\left(-\frac{s}{2},\frac{1}{2};\frac{z^2}{4}\right)(\sqrt{\pi}\alpha n)^{-s}\, ds\nonumber\\
&=:\sum_{n=1}^{\infty}J(z,\alpha,n)
\end{align}
for $0<\delta<1$. Note that shifting the line of integration from Re $s=1+\delta$ to Re $s=c$ does not introduce any poles.

Now (\ref{invmel1}) is valid for $0<\delta<1$. Replacing $s$ by $s+1$ in (\ref{invmel1}) and letting $x=\sqrt{\pi}\alpha n$, for $0<c=$ Re $s<1$, we have 
\begin{equation}\label{shotada}
\int_{c-i\infty}^{c+i\infty}\Gamma\left(\frac{s+1}{2}\right){}_1F_{1}\left(\frac{-s}{2};\frac{1}{2};\frac{z^2}{4}\right)x^{-s}\, ds
=4\pi ixe^{-x^2+z^2/4}\cos xz.
\end{equation}
Also, from (\ref{osinmel}), we easily see that for $0<d=$ Re $s<2$,
\begin{equation}\label{melosin}
\frac{1}{2\pi i}\int_{d-i\infty}^{d+i\infty}\frac{\pi}{\sin\tfrac{1}{2}\pi s}x^{-s}\, ds=\frac{2}{1+x^2}.
\end{equation}
Hence from (\ref{melconv}), (\ref{shotada}) and (\ref{melosin}), we see that for $0<c=$ Re $s<1$,
\begin{align}\label{osn}
J(z,\alpha,n)=8\pi ie^{\frac{z^{2}}{4}}\int_{0}^{\infty}\frac{e^{-x^2}\cos xz}{1+\left(\sqrt{\pi}\alpha n/x\right)^2}\, dx.
\end{align}
Thus from (\ref{resimpactr}), (\ref{lrc}), (\ref{calj0r}) and (\ref{osn}),
{\allowdisplaybreaks\begin{align}\label{osn1}
K(z,\alpha)&=4i\alpha^{-\frac{1}{2}}e^{-\frac{z^2}{8}}\left(8\pi ie^{\frac{z^{2}}{4}}\sum_{n=1}^{\infty}\int_{0}^{\infty}\frac{e^{-x^2}\cos xz}{1+\left(\sqrt{\pi}\alpha n/x\right)^2}\, dx-\frac{2\pi^{3/2}i}{\alpha}{}_1F_{1}\left(-\frac{1}{2},\frac{1}{2};\frac{z^2}{4}\right)\right)\nonumber\\
&=-8\pi^{\frac{3}{2}}\alpha^{-\frac{1}{2}}e^{-\frac{z^2}{8}}\left(4\alpha e^{\frac{z^2}{4}}\sum_{n=1}^{\infty}\int_{0}^{\infty}\frac{t^2e^{-\pi\alpha^2t^2}\cos\left(\sqrt{\pi}\alpha tz\right)}{t^2+n^2}\, dt-\frac{1}{\alpha}{}_1F_{1}\left(-\frac{1}{2},\frac{1}{2};\frac{z^2}{4}\right)\right)\nonumber\\
&=-8\pi^{\frac{3}{2}}\alpha^{-\frac{1}{2}}e^{-\frac{z^2}{8}}\left(4\alpha e^{\frac{z^2}{4}}\int_{0}^{\infty}t^2e^{-\pi\alpha^2t^2}\cos\left(\sqrt{\pi}\alpha tz\right)\left(\sum_{n=1}^{\infty}\frac{1}{t^2+n^2}\right)\, dt-\frac{1}{\alpha}{}_1F_{1}\left(-\frac{1}{2},\frac{1}{2};\frac{z^2}{4}\right)\right).\nonumber\\
\end{align}}%
For $t\neq 0$ \cite[p.~191]{con},
\begin{equation}\label{cotid1}
\sum_{n=1}^{\infty}\frac{1}{t^2+n^2}=\frac{\pi}{t}\left(\frac{1}{e^{2\pi t}-1}-\frac{1}{2\pi t}+\frac{1}{2}\right).
\end{equation}
Substituting (\ref{cotid1}) in (\ref{osn1}), we see that
\begin{align}\label{osn2}
K(z,\alpha)&=-8\pi^{\frac{3}{2}}\alpha^{-\frac{1}{2}}e^{-\frac{z^2}{8}}\bigg(4\pi\alpha e^{\frac{z^2}{4}}\int_{0}^{\infty}\frac{te^{-\pi\alpha^2t^2}\cos\left(\sqrt{\pi}\alpha tz\right)}{e^{2\pi t}-1}\, dt-2\alpha e^{\frac{z^2}{4}}\int_{0}^{\infty}e^{-\pi\alpha^2t^2}\cos\left(\sqrt{\pi}\alpha tz\right)\, dt\nonumber\\
&\quad\quad\quad\quad+2\pi\alpha e^{\frac{z^2}{4}}\int_{0}^{\infty}te^{-\pi\alpha^2t^2}\cos\left(\sqrt{\pi}\alpha tz\right)\, dt-\frac{1}{\alpha}{}_1F_{1}\left(-\frac{1}{2},\frac{1}{2};\frac{z^2}{4}\right)\bigg).\nonumber\\
\end{align}
But from \cite[p.~488, formula 3.896, no. 4]{grn},
\begin{equation}\label{1teg}
\int_{0}^{\infty}e^{-\pi\alpha^2t^2}\cos\left(\sqrt{\pi}\alpha tz\right)\, dt=\frac{e^{-\frac{z^2}{4}}}{2\alpha}.
\end{equation}
Also, from \cite[p.~502, formula 3.952, no. 2]{grn}, we have for $a>0$,
\begin{equation}\label{newser}
\int_{0}^{\infty}xe^{-p^2x^2}\cos (ax)\, dx=\frac{1}{2p^2}-\frac{a}{4p^3}\sum_{k=0}^{\infty}\frac{(-1)^kk!}{(2k+1)!}\left(\frac{a}{p}\right)^{2k+1}
\end{equation}
However, this formula holds for any $a\in\mathbb{C}$. Hence, letting $a=\sqrt{\pi}\alpha z$ and $p=\sqrt{\pi}\alpha$ in (\ref{newser}), simplifying the right-hand side, and then using (\ref{kft}), we have 
\begin{equation}\label{2teg}
\int_{0}^{\infty}te^{-\pi\alpha^2t^2}\cos\left(\sqrt{\pi}\alpha tz\right)\, dt=\frac{e^{-\frac{z^2}{4}}}{2\pi\alpha^2}{}_1F_{1}\left(-\frac{1}{2},\frac{1}{2};\frac{z^2}{4}\right).
\end{equation}
Substituting (\ref{1teg}) and (\ref{2teg}) in (\ref{osn2}), we have
\begin{align}\label{osn3}
K(z,\alpha)&=-8\pi^{\frac{3}{2}}\alpha^{-\frac{1}{2}}e^{-\frac{z^2}{8}}\left(4\pi\alpha e^{\frac{z^2}{4}}\int_{0}^{\infty}\frac{te^{-\pi\alpha^2t^2}\cos\left(\sqrt{\pi}\alpha tz\right)}{e^{2\pi t}-1}\, dt-1\right).
\end{align}
Finally, substituting (\ref{osn3}) in (\ref{kzac}), we obtain (\ref{rgenr}).
\section{Generalization of the Ramanujan-Hardy-Littlewood conjecture}
The approach here is similar to the one used by Hardy and Littlewood in \cite{hl} and so we will be brief. Shifting the line of integration from Re $s=1+\delta$, where $\delta>0$, to $-1<c=$ Re $s<0$ in the integral on the left-hand side of (\ref{invmel1}) and replacing $n$ by $1/n$, we have
\begin{equation}\label{invmel2}
\int_{c-i\infty}^{c+i\infty}\Gamma\left(\frac{s}{2}\right){}_1F_{1}\left(\frac{1-s}{2};\frac{1}{2};\frac{z^2}{4}\right)\left(\frac{\sqrt{\pi}\alpha}{n}\right)^{-s}\, ds=4\pi ie^{\frac{z^2}{4}}\left(e^{-\frac{\pi\alpha^2}{n^2}}\cos\left(\frac{\sqrt{\pi}\alpha z}{n}\right)-1\right).
\end{equation}
Note that one version of the prime number theorem reads $\sum_{n=1}^{\infty}\frac{\mu(n)}{n}=0$. Using these two facts in the calculation that follows, we have
\begin{align}\label{musim}
&\sqrt{\alpha}e^{\frac{z^2}{4}}\sum_{n=1}^{\infty}\frac{\mu(n)}{n}e^{-\frac{\pi\alpha^2}{n^2}}\cos\left(\frac{\sqrt{\pi}\alpha z}{n}\right)\nonumber\\
&=\sqrt{\alpha}e^{\frac{z^2}{4}}\sum_{n=1}^{\infty}\frac{\mu(n)}{n}\left(e^{-\frac{\pi\alpha^2}{n^2}}\cos\left(\frac{\sqrt{\pi}\alpha z}{n}\right)-1\right)\nonumber\\
&=\frac{\sqrt{\alpha}}{4\pi i}\sum_{n=1}^{\infty}\frac{\mu(n)}{n}\int_{c-i\infty}^{c+i\infty}\Gamma\left(\frac{s}{2}\right){}_1F_{1}\left(\frac{1-s}{2};\frac{1}{2};\frac{z^2}{4}\right)\left(\frac{\sqrt{\pi}\alpha}{n}\right)^{-s}\, ds\nonumber\\
&=\frac{\sqrt{\alpha}}{4\pi i}e^{\frac{z^2}{4}}\int_{c-i\infty}^{c+i\infty}\frac{\Gamma\left(\frac{s}{2}\right)}{\zeta(1-s)}{}_1F_{1}\left(\frac{1-s}{2};\frac{1}{2};\frac{z^2}{4}\right)\left(\sqrt{\pi}\alpha\right)^{-s}\, ds,
\end{align}
where we interchanged the order of summation and integration, which is valid because of absolute convergence, and we replaced $\sum_{n=1}^{\infty}\mu(n)n^{-(1-s)}$ by $1/\zeta(1-s)$ since Re $1-s>1$. Using (\ref{zetafe}) in (\ref{musim}), we have
\begin{equation}\label{musim1}
\sqrt{\alpha}e^{\frac{z^2}{4}}\sum_{n=1}^{\infty}\frac{\mu(n)}{n}e^{-\frac{\pi\alpha^2}{n^2}}\cos\left(\frac{\sqrt{\pi}\alpha z}{n}\right)=\frac{\sqrt{\alpha}}{4\pi^{\tfrac{3}{2} }i}\int_{c-i\infty}^{c+i\infty}\frac{\Gamma\left(\frac{1-s}{2}\right)}{\zeta(s)}{}_1F_{1}\left(\frac{1-s}{2};\frac{1}{2};\frac{z^2}{4}\right)\left(\frac{\alpha}{\sqrt{\pi}}\right)^{-s}\, ds.
\end{equation}
We want to shift the line of integration from $-1<c=$Re $s<0$ to Re $s=\lambda$, $\lambda\in(1,2)$, so that we can use the series representation $\sum_{n=1}^{\infty}\mu(n)n^{-s}$ for $1/\zeta(s)$. Consider a positively oriented rectangular contour formed by $[c-iT, \lambda-iT], [\lambda-iT, \lambda+iT], [\lambda+iT,c+iT]$ and $[c+iT,c-iT]$, where $T$ is any positive real number. In the shifting process, we encounter the non-trivial zeros of $\zeta(s)$. Hence upon the application of the residue theorem and assuming that the non-trivial zeros are simple, we get
\begin{align}\label{rhlgres}
&\int_{c-iT}^{c+iT}\frac{\Gamma\left(\frac{1-s}{2}\right)}{\zeta(s)}{}_1F_{1}\left(\frac{1-s}{2};\frac{1}{2};\frac{z^2}{4}\right)\left(\frac{\alpha}{\sqrt{\pi}}\right)^{-s}\, ds\nonumber\\
&=\left[\int_{c-iT}^{\lambda-iT}+\int_{\lambda-iT}^{\lambda+iT}+\int_{\lambda+iT}^{c+iT}\right]\frac{\Gamma\left(\frac{1-s}{2}\right)}{\zeta(s)}{}_1F_{1}\left(\frac{1-s}{2};\frac{1}{2};\frac{z^2}{4}\right)\left(\frac{\alpha}{\sqrt{\pi}}\right)^{-s}\, ds\nonumber\\
&\quad-2\pi i\sum_{-T<\rho<T}\lim_{s\to\rho}(s-\rho)\frac{\Gamma\left(\frac{1-s}{2}\right)}{\zeta(s)}{}_1F_{1}\left(\frac{1-s}{2};\frac{1}{2};\frac{z^2}{4}\right)\left(\frac{\alpha}{\sqrt{\pi}}\right)^{-s}.
\end{align}
Let $T\to\infty$ through values such that $|T-\gamma|>\exp\left(-A_{1}\gamma/\log\gamma\right)$ for every ordinate $\gamma$ of a zero of $\zeta(s)$, where $A_{1}$ is a positive constant. From \cite[p. 218, Equation (9.7.3)]{titch},
\begin{equation*}
\log|\zeta(\sigma+it)|\geq \sum_{|t-\gamma|\leq 1}\log|t-\gamma|+O(\log t).
\end{equation*}
Hence for $c<\sigma=$ Re $s<\lambda$,
\begin{equation*}
\log|\zeta(\sigma+iT)|\geq -\sum_{|T-\gamma|\leq 1}A_{1}\gamma/\log\gamma+O(\log T)>-A_{2}T,
\end{equation*}
where $A_{2}<\pi/4$ if $A_{1}$ is small enough, and $T>T_{0}$. This along with (\ref{strivert}) and (\ref{conasy}) implies that the integrals along the horizontal segments $[c-iT,\lambda-iT]$ and $[\lambda+iT, c+iT]$ tend to zero as $T\to\infty$ through the above values. It remains to evaluate
\begin{align}\label{jrhlg}
J(z,\alpha)&:=\int_{\lambda-i\infty}^{\lambda+i\infty}\frac{\Gamma\left(\frac{1-s}{2}\right)}{\zeta(s)}{}_1F_{1}\left(\frac{1-s}{2};\frac{1}{2};\frac{z^2}{4}\right)\left(\frac{\alpha}{\sqrt{\pi}}\right)^{-s}\, ds\nonumber\\
&=\sum_{n=1}^{\infty}\mu(n)\int_{\lambda-i\infty}^{\lambda+i\infty}\Gamma\left(\frac{1-s}{2}\right){}_1F_{1}\left(\frac{1-s}{2};\frac{1}{2};\frac{z^2}{4}\right)\left(\frac{\alpha n}{\sqrt{\pi}}\right)^{-s}\, ds\nonumber\\
&=:\sum_{n=1}^{\infty}\mu(n)J(z,\alpha, n).
\end{align}
Employing the change of variable $w=1-s$ in the integral in $J(z,\alpha, n)$, for $-1<\lambda'=$ Re $w<0$, we have
\begin{align}\label{jrhlg2}
J(z,\alpha,n)&=\frac{\sqrt{\pi}}{\alpha n}\int_{\lambda'-i\infty}^{\lambda'+i\infty}\Gamma\left(\frac{w}{2}\right){}_1F_{1}\left(\frac{w}{2};\frac{1}{2};\frac{z^2}{4}\right)\left(\frac{\sqrt{\pi}}{\alpha n}\right)^{-w}\, dw\nonumber\\
&=\frac{\sqrt{\pi} e^{\frac{z^2}{4}}}{\alpha n}\int_{\lambda'-i\infty}^{\lambda'+i\infty}\Gamma\left(\frac{w}{2}\right){}_1F_{1}\left(\frac{1-w}{2};\frac{1}{2};-\frac{z^2}{4}\right)\left(\frac{\sqrt{\pi}}{\alpha n}\right)^{-w}\, dw\nonumber\\
&=\frac{4\pi^{3/2}i\beta}{n}\left(e^{-\frac{\pi\beta^2}{n^2}}\cosh\left(\frac{\sqrt{\pi}\beta z}{n}\right)-1\right),
\end{align}
where in the penultimate step, we used (\ref{kft}) and in the last step, we used (\ref{invmel2}) with $z$ replaced by $iz$ and $\alpha$ replaced by $\beta$. Thus letting $T\to\infty$ in (\ref{rhlgres}) and combining with (\ref{musim1}), (\ref{jrhlg}) and (\ref{jrhlg2}), we obtain
\begin{align*}
&\sqrt{\alpha}e^{\frac{z^2}{4}}\sum_{n=1}^{\infty}\frac{\mu(n)}{n}e^{-\frac{\pi\alpha^2}{n^2}}\cos\left(\frac{\sqrt{\pi}\alpha z}{n}\right)\nonumber\\
&=-\sqrt{\beta}\sum_{n=1}^{\infty}\frac{\mu(n)}{n}\left(1-e^{-\frac{\pi\beta^2}{n^2}}\cosh\left(\frac{\sqrt{\pi}\beta z}{n}\right)\right)-\frac{\sqrt{\alpha}}{2\sqrt{\pi}}\sum_{\rho}\frac{\Gamma\left(\frac{1-\rho}{2}\right)}{\zeta^{'}(\rho)}{}_1F_{1}\left(\frac{1-\rho}{2};\frac{1}{2};\frac{z^2}{4}\right)\left(\frac{\alpha}{\sqrt{\pi}}\right)^{-\rho}.
\end{align*}
Using the prime number theorem again, we have
\begin{align}\label{0eq}
&\sqrt{\alpha}e^{\frac{z^2}{8}}\sum_{n=1}^{\infty}\frac{\mu(n)}{n}e^{-\frac{\pi\alpha^2}{n^2}}\cos\left(\frac{\sqrt{\pi}\alpha z}{n}\right)-\sqrt{\beta}e^{-\frac{z^2}{8}}\sum_{n=1}^{\infty}\frac{\mu(n)}{n}e^{-\frac{\pi\beta^2}{n^2}}\cosh\left(\frac{\sqrt{\pi}\beta z}{n}\right)\nonumber\\
&=-\frac{e^{-\frac{z^2}{8}}}{2\sqrt{\pi}\sqrt{\beta}}\sum_{\rho}\frac{\Gamma\left(\frac{1-\rho}{2}\right)}{\zeta^{'}(\rho)}{}_1F_{1}\left(\frac{1-\rho}{2};\frac{1}{2};\frac{z^2}{4}\right)\pi^{\rho/2}\beta^{\rho}.
\end{align}
We have not shown the convergence of the series in the above equation in the ordinary sense, but rather only when the terms are bracketed in such a way that two terms for which
\begin{equation*}
|\gamma-\gamma'|<\exp\left(-A_{1}\gamma/\log \gamma\right)+\exp\left(-A_{1}\gamma'/\log \gamma'\right)
\end{equation*}
are included in the same bracket. Replacing $\alpha\to\beta$ and $z\to iz$ in (\ref{0eq}), simplifying, and using (\ref{0eq}) again, we easily see that 
{\allowdisplaybreaks\begin{equation}\label{eq}
\frac{e^{\frac{z^2}{8}}}{2\sqrt{\pi}\sqrt{\alpha}}\sum_{\rho}\frac{\Gamma\left(\frac{1-\rho}{2}\right)}{\zeta^{'}(\rho)}{}_1F_{1}\left(\frac{1-\rho}{2};\frac{1}{2};-\frac{z^2}{4}\right)\pi^{\rho/2}\alpha^{\rho}+\frac{e^{-\frac{z^2}{8}}}{2\sqrt{\pi}\sqrt{\beta}}\sum_{\rho}\frac{\Gamma\left(\frac{1-\rho}{2}\right)}{\zeta^{'}(\rho)}{}_1F_{1}\left(\frac{1-\rho}{2};\frac{1}{2};\frac{z^2}{4}\right)\pi^{\rho/2}\beta^{\rho}=0.
\end{equation}}
Thus (\ref{0eq}) and (\ref{eq}) give (\ref{mrg}) upon simplification.
\section{Concluding remarks}
The integrands of all the integrals that we have considered here are not only continuous functions in both variables $t$ and $z$ but also analytic in $\mathbb{C}$ as a function of $z$ for each fixed value of $t$. Since all of these integrals converge uniformly at both limits in any compact subset of $\mathbb{C}$, from \cite[p.~31, Theorem 2.3]{temme}, we find that each of these integrals is holomorphic in $\mathbb{C}$ as a function of $z$ and that their derivatives of all orders may be found by differentiating under the sign of integration. This way, we can obtain many more transformation formulas or identities.

In this paper, we have focused on the generalization of the integral $\int_{0}^{\infty}f\left(\frac{t}{2}\right)\Xi(t)\cos\left(\tf{1}{2}t\log\alpha\right)\, dt$, where we generalized $\cos\left(\tf{1}{2}t\log\alpha\right)$ by making use of the function $\rho(\alpha,z,s)$ consisting of the confluent hypergeometric function. In \cite{riemann}, Ramanujan introduced the integral 
\begin{equation*}
\int_{0}^{\infty}\Gamma\left(\frac{z-1+it}{4}\right)\Gamma\left(\frac{z-1-it}{4}\right)
\Xi\left(\frac{t+iz}{2}\right)\Xi\left(\frac{t-iz}{2}\right)\frac{\cos\mu t}{(z+1)^2+t^2}\, dt,
\end{equation*}
where Re $z$ is not an integer and $\mu$ is real, and expressed it in terms of another integral in each of the regions \textup{Re} $s>1$, $-1<$ \textup{Re} $s<1$ and $-3<$ \textup{Re} $s<-1$. See also \cite[Theorem 1.2]{dixit}. In \cite{transf}, we examined the following general integral, with $\eta\in\mathbb{C}$,
\begin{equation}\label{lint}
\int_{0}^{\infty}f\left(\eta,\frac{t}{2}\right)\Xi\left(\frac{t+i\eta}{2}\right)\Xi\left(\frac{t-i\eta}{2}\right)\cos\left(\frac{1}{2}t\log\alpha\right)\, dt.
\end{equation}
It seems natural then to generalize the above integral by introducing in it a similar generalization of $\cos\left(\tf{1}{2}t\log\alpha\right)$ that we have utilized in this paper. Thus, replacing $\cos\left(\tf{1}{2}t\log\alpha\right)$ in (\ref{lint}) by either $\rho\left(\alpha,z,\frac{1+it}{2}\right)$ or its variants like $\rho\left(\alpha,z,\frac{3+it}{2}\right)$ (as in Theorem \ref{thmrgenr}) should lead to a generalization of the integrals in (\ref{lint}). The substitution $\rho\left(\alpha,z,\frac{1+it}{2}\right)$ will generate formulas of the type $F(\eta,z,\alpha)=F(\eta,iz,\beta)$, where $\alpha\beta=1$. Also, it looks plausible that either the substitution $\rho\left(\alpha,z,\frac{3+it}{2}\right)$ or $\rho\left(\alpha,z,\frac{1+it}{2}\right)$ should generalize Theorem 1.2 in \cite{dixit} found by Ramanujan. However, in both cases, the associated inverse Mellin transforms are difficult to evaluate. Our search in this direction did not lead to any particular nice example. Further efforts, however, may be fruitful and may result in beautiful and more general identities, for example, a generalization of the extended version of the Ramanujan-Guinand formula \cite[Theorem 1.4]{transf}. 

Another possible direction of generalizing the transformation formulas resulting from the integrals of the form $\int_{0}^{\infty}f\left(\frac{t}{2}\right)\Xi\left(\frac{t}{2}\right)\cos\left(\tf{1}{2}t\log\alpha\right)\, dt$ may be to replace the function $\rho(\alpha,z,s)$ used in this paper with an analogous one which involves the hypergeometric function ${}_2F_{2}$ instead of a ${}_1F_{1}$. This is because of the following Kummer-type transformation that exists for a ${}_2F_{2}$ \cite[Equation 4]{paris}:
\begin{equation*}
{}_2F_{2}\left(a,c+1;b,c;x\right)=e^{x}{}_2F_{2}\left(b-a-1,f+1;b,f;-x\right),
\end{equation*}
where
\begin{equation*}
f=\frac{c(1+a-b)}{a-c}.
\end{equation*}
In fact, a Kummer-type transformation also exists for the generalized hypergeometric function ${}_pF_{p}(x)$ \cite[Equation 4.2]{miller}.
\begin{center}
\textbf{Acknowledgements}
\end{center}
The author wishes to express his sincere thanks to Professor Bruce C.~Berndt for reading this manuscript in detail and for making helpful comments which improved the quality of this paper. 

\end{document}